\documentclass[11pt]{amsart}
\usepackage{amssymb}
\usepackage{tikz-cd}
\usetikzlibrary{decorations.pathmorphing}

\usepackage{amsrefs}

\title{Aspects of automatic continuity}

\author {Christian Rosendal and Luis Carlos Suarez}
\address{Department of Mathematics\\University of Maryland\\4176 Campus Drive - William E. Kirwan Hall\\College Park, MD 20742-4015\\USA}
\email{rosendal@umd.edu}

\urladdr{https://sites.google.com/view/christian-rosendal/}

\email{lcsuarez@umd.edu}

\newcommand{\norm}[1]{\lVert#1\rVert}

\newcommand{\acts}[1]{\mathop{\overset{#1}\curvearrowright}}
\newcommand{\mgd}[2]{\{  {#1}\;|\; {#2} \} }
\newcommand{\Mgd}[2]{\big\{  {#1}\;\big|\; {#2} \big\} }
\newcommand{\MGD}[2]{\Big\{  {#1}\;\Big|\; {#2} \Big\} }

\newcommand{\forkindep}[1][]{\mathop{\mathop{\vcenter{\hbox{\oalign{\noalign{\kern-.3ex}
\hfil$\vert$\hfil\cr\noalign{\kern-.7ex}$\smile$\cr\noalign{\kern-.3ex}}}}}\displaylimits_{#1}}}
\newcommand{\maps}[1]{\mathop{\overset{#1}\longrightarrow}}

\newcommand{\maths}[1]{\[\begin{split}{#1}\end{split}\]}
\newcommand {\N}{\mathbb N}
\newcommand {\Z}{\mathbb Z}
\newcommand {\Q}{\mathbb Q}
\newcommand {\R}{\mathbb R}
\newcommand {\C}{\mathbb C}

\newcommand {\U}{\mathbb U}

\newcommand{\normal}{\vartriangleleft}
\newcommand{\om}{\omega}
\newcommand{\eps}{\epsilon}

\newcommand{\iso}{\cong}

\newcommand{\tom} {\emptyset}

\newcommand{\saa}{\Rightarrow}
\newcommand{\equi}{\Leftrightarrow}

\newcommand {\del}{ \; \big| \;}

\newcommand {\go} {\mathfrak}
\newcommand {\ku} {\mathcal}

\newcommand{\ov}{\overline}
\newcommand{\inv}{^{-1}}

\newtheorem{thm}{Theorem}
\newtheorem{cor}[thm]{Corollary}
\newtheorem{lemme}[thm]{Lemma}

\newtheorem{prop} [thm] {Proposition}
\newtheorem{conj} [thm] {Conjecture}

\newtheorem{prob}[thm]{Problem}

\newtheorem{claim}[thm] {Claim}

\theoremstyle{definition}
\newtheorem{defi}[thm]{Definition}

\newtheorem{exa}[thm]{Example}

\definecolor{groen}{rgb}{0,0.5,.7}
\definecolor{gul}{rgb}{0.94,0.8,0}
\definecolor{blaa}{rgb}{0.16,0,0.6}
\definecolor{roed}{rgb}{1,0,0}

\makeindex

\thanks{The authors acknowledge  support by the U.S.~National Science Foundation under Award Numbers DMS-2246986 and DMS-2204849.  Thanks are also due to Paul Larson, Sang-hyun Kim and  Jindra Zapletal for interesting comments and helpful conversations.}

\begin{document}
\subjclass[2010]{Primary: 22A05, Secondary: 03E15, 03E25}
\keywords{Polish groups, automatic continuity, universally measurable sets, homeomorphism groups, isometry groups, consequences of the axiom of choice}

\maketitle

\begin{abstract}
A general overview of the phenomenon of automatic continuity of homomorphisms between Polish groups is given. In particular, we study variants and improvements of the closed graph theorem, applying these to the problem of continuity of universally measurable homomorphisms and also to gauge the amount of choice needed to construct discontinuous homomorphisms between Polish groups. Furthermore, we provide a simple proof of automatic continuity in the context of homeomorphism groups of compact manifolds and a complete reworking of automatic continuity theory in the context of isometry groups of highly homogeneous complete metric structures.
\end{abstract}

\section{Introduction}
The present paper provides an overview of various new directions in automatic continuity theory. More concretely, the general problem considered here is the conditions under which a homomorphism between two topological groups or, more specifically, Polish groups is continuous. To set the stage, let us recall that a {\em Polish group} is a topological group whose underlying topology is {\em Polish}, that is, separable and completely metrisable. These groups provide a useful generalisation of the class of locally compact, second countable groups, and include most of the classically studied topological transformation groups such as Lie groups, homeomorphism and diffeomorphism groups of manifolds, isometry groups of separable complete metric spaces, automorphism groups of countable discrete structures, and totally disconnected locally compact groups. 

Our main issue here is one of topological rigidity. Namely, what are the conditions on Polish groups $G$ and $H$ that ensure that every homomorphism 
$$
G\maps \pi H
$$
is continuous? Or what abstract conditions on homomorphisms $\pi$ between Polish groups ensure that $\pi$ is continuous? There are several other related concepts of topological rigidity studied in the literature such as algebraic simplicity,  uniqueness and minimality of Polish group topologies, and the Bergman property  (also known as uncountable strong cofinality). However, with the exception of Section \ref{sec:quadri}, where small index properties are compared with automatic continuity, and Section \ref{section:isometry}, where we discuss isometry groups of complete metric structures, we will not be substantially concerned with any of these, although several of the techniques developed here are applicable to these other concepts as well.

So what is the interest of automatic continuity or of any of the other manifestations of topological rigidity? Simply put, it allows one to deal with various topological-algebraic objects without regard to their topological structure and thus occasionally retain a simple algebraic theory that might otherwise require more complex assumptions. To take a particular illustrative example, for a general group $G$, let 
$$
{\sf Hom}(G,\Z)=\{\phi \colon G\to \Z\del \phi \text{ is a group homomorphism }\}
$$
and observe that ${\sf Hom}(G,\Z)$ is itself an abelian group under the operation defined by $(\phi+\psi)(g)=\phi(g)+\psi(g)$. Let also $\prod_{n=1}^\infty \Z$ be full direct product of infinitely many copies of $\Z$ viewed as the set of integer valued sequences with coordinatewise addition and let $\bigoplus_{n=1}^\infty \Z$ be the subgroup of $\prod_{n=1}^\infty \Z$ consisting of finitely supported such sequences. It is now easy to see that
$$
{\sf Hom}\Bigg(\bigoplus_{n=1}^\infty \Z\,,\,\Z\Bigg)=\prod_{n=1}^\infty \Z,
$$
as every homomorphism $\phi\colon\bigoplus_{n=1}^\infty \Z\to\Z$ has the form 
$$
\phi(a_1,a_2,\ldots)=\sum_{n=1}^\infty z_na_n
$$
for some sequence $(z_n)\in \prod_{n=1}^\infty \Z$. What is remarkable however is that, conversely,
$$
{\sf Hom}\Bigg(\prod_{n=1}^\infty \Z\,,\,\Z\Bigg)=\bigoplus_{n=1}^\infty \Z.
$$
This fact is due to E. Specker \cite{specker}, who showed that every homomorphism  $$\psi\colon\prod_{n=1}^\infty \Z\to\Z$$
only depends on finitely many coordinates or, equivalently, is continuous with respect to the discrete topology on $\Z$ and the Tychonoff product topology on $\prod_{n=1}^\infty \Z$. In other words, $\psi$ necessarily has the form
 $$
\psi(z_1,z_2,\ldots)=\sum_{n=1}^\infty a_nz_n
$$
for some  sequence $(a_n)\in \bigoplus_{n=1}^\infty \Z$. 

For another example,  consider the homomorphism 
$$
{\sf Homeo}({\mathbb D^2})\maps{\rho}{\sf Homeo}(\mathbb S^1)
$$
induced by the restriction  of a homeomorphism of the closed disk ${\mathbb D^2}$ to the boundary $\mathbb S^1$. It is easy to conversely construct an inverse $\varepsilon$ to $\rho$, that is, an extension homomorphism
$$
{\sf Homeo}(\mathbb S^1)\maps{\varepsilon}{\sf Homeo}({\mathbb D^2}), 
$$
but D. Epstein and V. Markovic \cite{markovic} asked if all such extension homomorphisms are necessarily continuous. As it turns out, much more is true, since all homomorphisms from ${\sf Homeo}(\mathbb S^1)$ with values in a separable topological group are continuous \cite{ros-sol}. This latter result along with its generalisation to homeomorphism groups of all compact manifolds in \cites{rosendal-israel,mann} also provide the basis for the explicit description of actions of the identity component ${\sf Homeo}_0(M)$ of the homeomorphism group of a topological manifold $M$ on another manifold $N$ given by L. Chen and K. Mann in \cite{chen-mann}.

The organisation of our paper is as follows. Section \ref{sec:closed graph} provides an account of the many facets of the closed graph theorem, a result that provides a weak sufficient condition for the continuity of a homomorphism. The analysis of this classical result also leads to the introduction of a {\em characteristic group} $N$ associated with any given homomorphism $G\maps \pi H$ between Polish groups, whose size is a measure of the discontinuity of $\pi$. In Section \ref{sec:christensen}, the characteristic group is put to use in the context of an old problem of J.P.R. Christensen \cite{christensen} concerning the continuity of universally measurable homomorphisms between Polish groups. We present the positive solution to this problem from \cite{rosendal-Pi} and also the further developments from \cite{banakh2022automatic}. As is well-known from results in set theory, some usage of the axiom of choice is needed in order to produce discontinuous homomorphisms between Polish groups, and in Section \ref{sec:quadri} we present a quadrichotomy that in combinatorial terms calibrates the exact amount of choice needed to produce homomorphisms with varying degrees of discontinuity as measured by the size of the associated characteristic group. In Section \ref{sec:homeo} we provide a greatly simplified proof of the automatic continuity results of \cite{ros-sol, rosendal-israel,mann} on homeomorphism groups of compact manifolds.  Similarly, in Section \ref{section:isometry} we provide a common framework due to M. Sabok \cite{sabok} for dealing with automatic continuity in the automorphism groups of highly homogeneous metric structures. Whereas our account relies heavily on the work of Sabok along with some simplifications introduced in \cite{Malicki}, we are able to discard of much excessive baggage  and give what is hopefully a more accessible and transparent exposition of these results. In particular, Corollary \ref{cor:autconti} provides a single and easily verifiable criterion for automatic continuity in this context. In turn, this criterion is applied to the isometry group of the Urysohn metric space in Section \ref{section:isometry}, providing an alternative proof of a result of Sabok \cite{sabok}, the automorphism group of the standard measure algebra in Section \ref{sec:meas}, recovering a result of I. Ben Yaacov, A. Berenstein and J. Melleray \cite{BBM}, and finally we use our methods to also give a proof for the unitary group of infinite-dimensional separable Hilbert space in Section \ref{sec:hilbert}, which is originally due to T. Tsankov \cite{tsankov}.

The paper \cite{rosendal-BSL} provided an account of the theory as it was about a decade and a half ago and we have largely avoided any overlap with that paper, except for the cases where progress has been made on a specific topic since  (for example, regarding universally measurable homomorphisms). The reader seeking a more complete account of the subject may thus also consult \cite{rosendal-BSL}, for example, regarding constructions of discontinuous homomorphisms, measurability assumptions and ample generics.  Finally, it should be stressed that a certain selection of topics has been made mostly having to do the  authors' specific interests and expertise. The literature is much richer than can be conveyed in a short account such as this. In particular, we are leaving out a good deal of recent literature dealing with identifying discrete groups $\Gamma$ so that every homomorphism $G\maps\pi \Gamma$ from a locally compact or completely metrisable topological group $G$ into $\Gamma$ is continuous. For example,  \cite{dudley,alperin,slutsky,bogopolski} contain some of the many highlights in this direction.

\section{Open and closed mapping theorem}\label{sec:closed graph}
If $X$ is a topological space, a subset $A\subseteq X$ is said to be {\em meagre} provided that $A=\bigcup_{n=1}^\infty F_n$ for some sequence of nowhere dense subsets $F_n\subseteq X$. Also, a set $B\subseteq X$ is {\em comeagre} provided that $A=X\setminus B$ is meagre or, equivalently, $B$ contains the intersection $\bigcap_{n=1}^\infty U_n$ of  dense open subsets $U_n\subseteq X$. Recall also that a topological space $X$ is said to be a {\em Baire} space provided that $X$ satisfies the following two equivalent conditions.
\begin{itemize}
\item Meagre subsets of $X$  have empty interior,
\item every comeagre set is dense in $X$.
\end{itemize}
Finally, a subset $A\subseteq X$ is said to have the {\em property of Baire} or to be {\em Baire measurable} if there is an open set $U\subseteq X$ so that $A\triangle U$ is meagre. To discuss this property, it is useful to introduce a bit of notation. Namely, for a subset $A$ of a topological space $X$, we let 
$$
U(A)=\bigcup\{ U \subseteq X\del U \text{ is open and $U\setminus A$ is meagre}\}.
$$
Then $U(A)\setminus A$ is also meagre \cite[Theorem 8.29]{kechris-book}. Observe also that the operator $U(\cdot)$ is determined by the topology of $X$ and therefore commutes with homeomorphisms, that is, if $f\colon X\to Y$ is a homeomorphism and $A\subseteq X$, then $U(f[A])=f[U(A)]$.

When dealing with topological groups, being a Baire space simply reduces to non-meagreness of the whole space.

\begin{lemme}
A topological group $G$ is a Baire space if and only if $G$ is non-meagre.
\end{lemme}

\begin{proof}
Observe that ${\sf int}\, G=G\neq \tom$ and so, if $G$ is a Baire space, then $G$ cannot itself be meagre. 

Conversely, suppose $G$ fails to be a Baire space. This means that there is a non-empty meagre open set $V\subseteq G$ and thus $V\subseteq U(\emptyset)$.
However, for every $x\in G$, the left translate $xV$ is also meagre and so $xV\subseteq U(\tom)$. We thus conclude that $G=\bigcup_{x\in G}xV\subseteq U(\tom)$, i.e., that $G$ is itself meagre. 
\end{proof}

\begin{lemme}[Pettis' lemma]\label{pettis}
Suppose $G$ is a Baire topological group and $A,B\subseteq G$ are subsets. Then
$$
U(A)\cdot U(B)\subseteq {\sf int}(AB).
$$
\end{lemme}

\begin{proof}
We note that if $x\in U(A)U(B)$, then the open set 
$$
V= x U(B)\inv\cap U(A)=U(xB\inv)\cap U(A)
$$ 
is non-empty and so $xB\inv$ and $A$ are comeagre in $V$.
It follows that also $xB\inv \cap A$ is non-meagre in $V$ and therefore is non-empty, whereby $x\in AB$. As $U(A)U(B)$ is itself open, it follows that $U(A)U(B)\subseteq {\sf int}(AB)$. 
\end{proof}

\begin{thm}[Open mapping theorem]
Suppose $G\maps\pi H$ is a continuous epimorphism between Polish groups. Then $\pi$ is an open map.
\end{thm}

\begin{proof}
We first show that, if $V$ is an identity neighbourhood in $G$, then $\pi[V]$ is an identity neighbourhood in $H$. To see this, fix $V$ and let $\{x_n\}_{n=1}^\infty$ be a countable dense subset in $G$. Pick also some open identity neighbourhood $W\subseteq G$ so that $W\inv W\subseteq V$. Then $G=\bigcup_{n=1}^\infty x_nW$ and thus
$H=\bigcup_{n=1}^\infty \pi(x_n)\pi[W]$. Because $H$ is not itself meagre, it follows that some  $\pi(x_n)\pi[W]$ is non-meagre and thus that the analytic set $\pi[W]$ is also non-meagre. Being analytic, this means that  $U\big(\pi[W]\big)\neq \tom$ \cite[Theorem 21.6]{kechris-book} and so 
\maths{
1
&\in U\big(\pi[W]\big)\inv \cdot U\big(\pi[W]\big)\\
&=U\big(\pi[W]\inv\big) \cdot U\big(\pi[W]\big)\\
&\subseteq {\sf int}\big(\pi[W]\inv\cdot \pi[W]\big)\\
&\subseteq {\sf int}\big(\pi[W\inv W]\big)\\
&\subseteq {\sf int}\big(\pi[V]\big)
}
by Lemma \ref{pettis}.

Suppose now $O$ is any non-empty open set in $G$. To see that $\pi[O]$ is open, we show that $\pi(x)\in {\sf int}\,\pi[O]$ for all $x\in O$. So let $x\in O$ be given and find an identity neighbourhood $V$ in $G$ so that $xV\subseteq O$. Then $\pi[V]$ is an identity neighbourhood in $H$ and $\pi(x)\pi[V]$ is a neighbourhood of $\pi(x)$ contained in $\pi[O]$, that is, $x\in {\sf int}\, \pi[O]$. 
\end{proof}

\begin{cor}[Closed graph theorem]
Let $G\overset{\pi}{\longrightarrow}H$ be a homomorphism between Polish groups. Then $\pi$ is continuous if and only if the graph $\ku G\pi$ is a closed subset of $G\times H$.
\end{cor}

\begin{proof}The graph of any continuous map is closed. So suppose instead  that $\ku G\pi$ is closed.
Observe that the graph $\ku G\pi$ is a subgroup of $G\times H$, since
$$
(g,\pi(g))\cdot (f,\pi(f))\inv=(gf\inv, \pi(g)\pi(f)\inv)=(gf\inv ,\pi(gf\inv))\in \ku G\pi
$$
for all $g,f\in G$. Being also closed, $\ku G\pi$ is a Polish group in the topology induced from $G\times H$. Moreover, $\ku G\pi$ is isomorphic with $G$ as an abstract group via the homomorphism $\ku G\pi \maps\phi  G$ defined by $\phi(g,h)=g$. Because $\phi$ is clearly continuous, being also surjective, it is an open mapping and so 
$$
\phi\big[\ku G\pi\cap (G\times V)\big]=\{g\in G\del (g,\pi(g))\in G\times V\}=\pi\inv (V)
$$
is open in $G$ for all open $V\subseteq H$. In other words, $\pi$ is continuous.
\end{proof}

\begin{defi}[The characteristic group of a homomorphism]
Let $G\overset{\pi}{\longrightarrow}H$ be a (possibly discontinuous) group homomorphism between two Polish groups. We define a closed subgroup of $H$, termed the {\em characteristic group} of $\pi$, by the formula
\begin{align*}
    N= \bigcap_V \overline{\pi[V]},
\end{align*}
where the intersection is taken over all identity neighbourhoods $V$ in $G$. Although $N$ depends on the homomorphism $\pi$, to simplify notation, we will simply denote it by $N$. The dependence on $\pi$ will be clear from the context.
\end{defi}
For other descriptions of $N$, observe that an element $h\in H$ belongs to $N$ if and only if
$$
h=\lim_{n}\pi(g_n)
$$
for some sequence $(g_n)$ in $G$ with $\lim_ng_n= 1$. It follows that 
\begin{align*}
N = \big\{h \in H \del (1, h) \in \overline{\ku G\pi}\,\big\}.
\end{align*}
Note also that, if $V$ and $W$ are identity neighbourhoods in $G$ so that $VV\inv \subseteq W$, then 
$$
\ov{\pi[V]}\cdot \ov {\pi[V]}\inv \subseteq \ov{\pi[V]}\cdot \ov {\pi[V\inv ]}\subseteq \ov{\pi[VV\inv ]}\subseteq \ov{\pi[W]}.
$$
This shows that $NN\inv \subseteq N$ and so, as $1\in N$, we see that  $N$ is a subgroup of $H$.

\begin{prop}
Let $G$ and $H$ be Polish groups and $G\maps \pi H$ a homomorphism  with dense image. Then the characteristic group $N$  is the smallest closed normal subgroup of $H$ so that the induced homomorphism $G\maps{\tilde{\pi}}H/N$ is continuous.
\end{prop}

\begin{proof}
Note first that, since $N$ is closed, the quotient group $H/N$ equipped with the quotient topology is a Polish group. To see that $\tilde \pi$ is continuous, it suffices, by the closed graph theorem to show that its graph is closed.

Observe that $\ku G \tilde{\pi}$ is a subgroup of $G \times H/N$ and hence so is its closure $\overline{\ku G\tilde{\pi}}$. Therefore, if $(g,hN) \in \overline{\ku G\tilde{\pi}}\setminus \ku G\tilde{\pi}$, also $(1, \pi(g)^{-1}hN) \in \overline{\ku G\tilde{\pi}}\setminus \ku G \tilde{\pi}$. Thus, to see that $\ku G \tilde{\pi}$ is closed, it suffices to show that $(1, fN) \notin \overline{\ku G\tilde{\pi}}$ provided $f \notin N$.

So assume that $f \notin N$. Then there is an open identity neighbourhood $U \subseteq G$ so that $f \notin \overline{\pi[U]}$ and thus also an open neighbourhood $V$ of $f$ with $V \cap \overline{\pi[U]}=\emptyset$.  Note that,  if  $u \in U$, we may pick an identity neighbourhood $W$ in $G$ so that $uW \subseteq U$, whereby
\begin{align*}
\pi(u)N \subseteq \pi(u) \overline{\pi[W]} = \overline{\pi(uW)}\subseteq \overline{\pi[U]}.
\end{align*}
This shows that $\pi[U]N \subseteq \overline{\pi[U]}$, which in turn implies that
$$
VN \cap \pi[U]N = \emptyset.
$$ 
Therefore, $U \times V/N$ is a neighbourhood of $(1,fN)$ disjoint from $\ku G \tilde{\pi}$ and we find that $(1,fN) \notin \overline{\ku G\tilde{\pi}}$ as required.

Now suppose that $K \subseteq H$ is any closed normal subgroup such that the induced homomorphism $G\overset{\tilde{\pi}}{\longrightarrow}H/K$ is continuous and let $h \in N$. Then we may find a sequence $(g_n)_n$ in $G$ such that $g_n \rightarrow 1$ and $\pi(g_n) \rightarrow h$, whereby also $\pi(g_n)K \rightarrow hK$. This shows that $(1, hK)$ lies in the closure of the graph of the continuous homomorphism $G\overset{\tilde{\pi}}{\longrightarrow}H/K$ and hence in the graph itself, whence $h\in K$. So $N\subseteq K$ and we see that $N$  is the smallest closed normal subgroup of $H$ so that the induced homomorphism $G\maps{\tilde{\pi}}H/N$ is continuous.
\end{proof}
As we are interested in the continuity of $\pi$, by replacing $H$ with the closed subgroup $\ov{\pi[G]}$, we can always assume that $\pi[G]$ is dense in $H$ and thus that $N$ is  normal in $H$. 

\begin{cor}\label{cor:closed graph}
A homomorphism $G\overset{\pi}{\longrightarrow}H$ between Polish groups is continuous if and only if $\bigcap_{U}\ov{\pi[U]}=\{1\}$, where the $U$ vary over identity neighbourhoods in $G$.
\end{cor}

Let us note that Corollary \ref{cor:closed graph} gives an a priori weaker condition for continuity of $\pi$ than requiring continuity at $1\in G$. Namely, 
$G\overset{\pi}{\longrightarrow}H$ is continuous at $1$ if the sets $\ov{\pi[U]}$ form a neighbourhood basis at $1$ when $U$ varies over identity neighbourhoods in $G$. In Corollary \ref{cor:closed graph}, we only require their intersection to be $\{1\}$.

A topological group $H$ is said to have {\em no small subgroups} or to be {\em NSS} if there is an identity neighbourhood $V\subseteq H$ so that the only subgroup of $H$ that is contained in $V$ is the trivial one, $\{1\}$. By  the solution to Hilbert's 5th problem, due to D. Mongomery, A. Gleason, H. Yamabe and L. Zippin \cite{tao}, a locally compact group is NSS if and only if it is a Lie group. The closed graph theorem now applies to give us a simple criterion for continuity of homomorphisms into Lie groups.

\begin{cor}\label{cor:closed graph4}
Suppose that $H$ is a Polish group and that $V\subseteq H$ is a closed identity neighbourhood containing no non-trivial subgroups. Then a homomorphism $G\overset{\pi}{\longrightarrow}H$ from a Polish group $G$ into $H$ is continuous if and only if $\pi\inv(V)$ is an identity neighbourhood in $G$. 
\end{cor}

\begin{proof}
    Observe that, if $G\overset{\pi}{\longrightarrow}H$ is a homomorphism so that $U=\pi\inv(V)$ is an identity neighbourhood in $G$, then
    $$
    N\subseteq \ov{\pi[U]}=\ov V=V
    $$
    and hence $N=\{1\}$. By Corollary \ref{cor:closed graph}, $\pi$ is continuous.
\end{proof}

\begin{exa}[Rigidity of continuity in the space of homomorphisms]
    Suppose $G$ is a locally compact Polish group and $H$ is a second countable Lie group, i.e., a Polish Lie group. Let also 
    $$
    {\sf Hom}(G,H)
    $$ 
    denote the space of all (potentially discontinuous) homomorphisms from $G$ to $H$ equipped with the topology of uniform convergence on compacta. That is, a net $(\pi_i)$ of homomorphisms converges to some homomorphism $\pi$ if, for every compact set $K\subseteq G$ and every identity neighbourhood $W\subseteq H$, we have that
    $$
    \pi_i(g)\in \pi(g) W \;\;\text{for all }g\in K
    $$
    for all sufficiently large $i$. 
    
    We claim that the collection of continuous homomorphisms is open in ${\sf Hom}(G,H)$. Indeed, assume $\pi$ is continuous and that $V$ is a closed identity neighbourhood in $H$ with no small subgroups. Pick some identity neighbourhood $W\subseteq H$ so that $WW\inv \subseteq V$ and let $K\subseteq G$ be a compact identity neighbourhood contained in $\pi\inv(W)$. Then, for any homomorphism $\pi'$ satisfying 
    $$
    \pi'(g)\in \pi(g) W \;\;\text{for all }g\in K
    $$
    we have that 
    $$
    \pi'[K]\subseteq \pi[K]W\subseteq W^2\subseteq V,
    $$
    whereby $\pi'$ is continuous by Corollary \ref{cor:closed graph4}.
\end{exa}

\begin{cor}\label{cor:closed graph2}
Let $G\overset{\pi}{\longrightarrow}H$ be a homomorphism between Polish groups and suppose $H\overset{\iota}{\longrightarrow}\Omega$ is a continuous injection into a Hausdorff topological space $\Omega$. Then   $\pi$ is continuous if and only if $\iota\circ \pi$ is continuous. 
\end{cor}

\begin{proof}If $\pi$ is continuous, then clearly so is the composition $\iota\circ\pi$. Conversely, suppose that $\iota\circ\pi$ is continuous. Then, by the continuity of both $\iota$ and $\iota\circ\pi$, we have 
$$
\iota\Big(\bigcap_{U}\ov{\pi[U]}\Big) \;\subseteq\; \bigcap_{U}\iota\Big(\ov{\pi[U]}\Big)\;\subseteq\; \bigcap_{U}\ov{\iota\big(\pi[U]\big)}=\{\iota(1)\}. 
$$
So, by injectivity of $\iota$, we find that
$$
\bigcap_{U}\ov{\pi[U]}=\{1\},
$$
whereby $\pi$ is continuous  by Corollary \ref{cor:closed graph}. 
\end{proof}

Let us give three instances of Corollary \ref{cor:closed graph2} of which the first two are well-known. Namely, suppose $X\maps TY$ is a linear operator between two (separable) Banach spaces and let $Y\maps{\iota}\prod_{Y^*}\R$ be the continuous injection defined by 
$$
\iota(y)\big(\phi\big)=\phi(y).
$$
By Corollary \ref{cor:closed graph2}, we see that $T$ is bounded, i.e., continuous, if and only if $\iota\circ T$ is continuous, which, by definition of the Tychonoff product topology, holds if and only if $\phi\circ T$ is continuous for every $\phi\in Y^*$. In this case, the separability of $X$ and $Y$ is irrelevant. 

Also, suppose $(X,d)$ is a separable complete metric space and $G$ is a Polish group acting by isometries on $(X,d)$. Then the action $G\times X\to X$ is jointly continuous if and only if the associated homomorphism $G\maps\pi {\sf Isom}(X,d)$ is continuous with respect to the Polish group topology on  ${\sf Isom}(X,d)$, which is the topology of pointwise convergence on $X$. Letting ${\sf Isom}(X,d)\maps \iota \prod_XX$ be defined by $\iota(f)_x=f(x)$, we see as in the previous example that the action $G\times X\to X$ is jointly continuous if and only if the orbit evaluation map
$$
g\in G\mapsto gx \in X
$$
is continuous for every $x\in X$. 

For the third less familiar instance, consider a compact smooth manifold $M$ and let for each $r=0,1,2,\ldots, \infty$
$$
{\sf Diff}^r(M)
$$
denote the Polish group of $C^r$-diffeomorphisms of $M$. Thus, for $r=0$, this is just the homeomorphism group ${\sf Homeo}(M)$ with the topology of uniform convergence, whereas for $r\geqslant 1$ not only is the group a proper subgroup of ${\sf Homeo}(M)$, but the topology is also finer. As in the preceding two examples, we have the following corollary. 

\begin{cor}\label{cor:closed graph3}
Let $G$ be a Polish group and $r=0,1,2,\ldots,\infty$. Then a homomorphism
$$
G\maps \pi {\sf Diff}^r(M)
$$
is continuous if and only the orbit evaluation map
$$
g\in G\;\mapsto\; \pi(g)x\in M
$$
is continuous for every $x\in M$. 
\end{cor}
Corollary \ref{cor:closed graph3} is evidently related to a well-known theorem of S. Bochner and D. Montgomery \cite{bochner} (see also \cite{kim} for a recent exposition). Their theorem says that, when  
$$
G\maps \pi {\sf Diff}^r(M)
$$
is a homomorphism from a Lie group $G$ into a $C^r$-manifold $M$ with $r=0,1,2,\ldots,\infty$ such that the action map 
$G\times M\to M$ is jointly continuous, then it is jointly $C^k$. Coupled with Corollary \ref{cor:closed graph3}, we see that it is enough to assume that the action map is separately continuous, but the corollary also provides additional information for other Polish groups.

From the results above, we remark that the study of the continuity of $G\overset{{\pi}}{\longrightarrow}H$ can be seen as a problem on the complexity of the characteristic group $N$. In particular, allowing $N$ to be non-trivial, gives us weaker notions of continuity. So let us focus our attention on the study of $N$.

We first recall a result of independent interest. Here the equivalence between (1) and (2) is classical, whereas the equivalence with (3) is due to S. Solecki \cite{solecki} and V. Uspenski\u\i{} \cite{uspenskii}  independently. 
\begin{thm}\label{thm:compactness} 
The following conditions are equivalent for a Polish group $G$.
\begin{enumerate}
    \item $G$ is compact
    \item for every identity neighbourhood $U$, there is a finite set $F$ so that $G=FU$,
    \item for every identity neighbourhood $U$, there is a finite set $F$ so that $G=FUF$.
\end{enumerate}
\end{thm}

\begin{proof}
(3) $\Rightarrow$ (2): Let $U$ be any given identity neighbourhood. Find an identity neighbourhood $V$ such that $VV^{-1} \subseteq U$ and let $f_1, \dots, f_n \in G$ be a sequence of minimal length for which there is a finite set $E \subseteq G$ so that
\begin{align*}
            G = \bigcup_{i=1}^n EVf_i.
        \end{align*}
        Fix such a set $E$ and by way of contradiction, assume that $G \neq EVV^{-1}$. Then we can pick $x \notin EVV^{-1}$, whereby $x V \cap EV=\emptyset$ and hence
        \begin{align*}
            xVf_n \subseteq G \setminus EVf_n \subseteq \bigcup_{i=1}^{n-1} EVf_i.
        \end{align*}
        It follows that 
        \begin{align*}
            EVf_n =Ex\inv x Vf_n\subseteq \bigcup_{i=1}^{n-1} Ex^{-1}EVf_i
        \end{align*}
        and therefore that
        \begin{align*}
            G=\bigcup_{i=1}^{n-1} (E\cup Ex^{-1}E)Vf_i,
        \end{align*}
        contradicting the minimality of $n$. We thus conclude that $G=EVV^{-1} \subseteq EU$, where $E$ is finite.

(2) $\Rightarrow$ (1): Let $d$ be a be a compatible left-invariant metric on $G$ and define a metric $D$ on $G$ by 
\begin{align*}
        D(g,f)= d(g,f)+d(g^{-1},f^{-1}).
\end{align*}
Whereas $D$ need not be left or right invariant, it is still compatible with the topology on $G$. Thus, if $\hat G$ denotes the completion of $G$ with respect to $D$, we see that $G$ is a Polish and hence $G_\delta$ subset of the Polish space $\hat G$. Being also dense, it follows that $G$ is a comeagre subset of $\hat G$.

Also, the group operations in $G$ extend continuously to the completion $\hat G$, meaning that $\hat G$ is a Polish group (see \cite{Rogers}, pp. 352-353). In other words, $G$ is a comeagre subgroup of the Polish group $\hat G$, implying that all its cosets in $\hat G$ are also comeagre and therefore intersect $G$, which is possible only if the index $[\hat G:G]$ is one. It follows that $G=\hat G$ and thus that $D$ is a compatible complete metric on $G$. To see that $G$ is compact, it therefore suffices to see that every sequence $(g_n)$ in $G$ admits a $D$-Cauchy subsequence.

So let $(g_n)$ be given and fix a neighbourhood basis $\{U_i\}_{i=1}^\infty$ at the identity. Choose also finite sets $F_i, E_i\subseteq G$ so that 
$$
G=F_iU_i=U_iE_i
$$
for all $i$. By a diagonal extraction procedure, we may produce a subsequence $(g_{k_n})$ of $(g_n)$ and elements $f_i\in F_i$, $h_i\in E_i$ so that $g_{k_n}\in f_iU_i\cap U_ih_i$ whenever $n\geqslant i$. It follows that, for all $i$ and all $n,m\geqslant i$, we have 
$$
g_{k_m}\inv g_{k_n}\in U_i\inv f_i\inv f_iU_i=U\inv_iU_i
$$
and similarly $g_{k_m}g_{k_n}\inv \in U_iU_i\inv$. It follows that both $(g_{k_n})$ and $(g_{k_n}\inv)$ are $d$-Cauchy, whereby $(g_{k_n})$ is  $D$-Cauchy as required.

(1) $\Rightarrow$ (3): Suppose $G$ is compact and let $U$ be an open identity neighbourhood. Then $G=\bigcup_{f\in G} fU$ is an open covering of $G$ and by compactness admits a finite subcovering $G=\bigcup_{i=1}^n f_i U$. Letting $F=\{1,f_1,\dots, f_n\}$, we have $G =FUF$.
\end{proof}

\begin{lemme}\label{lemme:trivialNbd} 
Let $N$ be a Polish group. Then $N=\{1\}$ if and only if, for every identity neighbourhood $V$, there is a finite set $F\subseteq N$ so that $N=\bigcup_{f \in F} fVf^{-1}$. 
\end{lemme}

\begin{proof}
The implication from left to right is obvious, so consider the other implication and assume that, for every identity neighbourhood $V$, there is a finite set $F\subseteq N$ so that $N=\bigcup_{f \in F} fVf^{-1}$. By Theorem \ref{thm:compactness}, this implies that, for any identity neighbourhood $V$, there is a finite set $E$ such that $N=EV$.   

We now show that for any symmetric neighbourhood $U$ of the identity, one has $N=U^3$. Since $N$ is Hausdorff, this shows that $N=\{1\}$. In order to see this, let $E$ be a finite set such that $N=EU$ and set $V=\bigcap_{g \in E}gUg^{-1}$. Let also $F$ be a finite set so that $N=\bigcup_{f \in F} fVf^{-1}$. For $f \in F$, write $f^{-1}=gu$ for some $g \in E$ and $u \in U$ and observe that
\begin{align*}
fVf^{-1} = u^{-1}g^{-1}Vgu \subseteq u^{-1}Uu \subseteq U^3.
\end{align*}
So $N=\bigcup_{f \in F}fVf^{-1} = U^3$.
\end{proof}

With the preparatory material in hand, we are now ready to prove two theorems that characterise the extent of continuity of a homomorphism $G\maps \pi H$ in terms of the size of the associated characteristic group $N$. The first can be viewed as an extension of the closed graph theorem.

\begin{thm}\cite{rosendal-Pi}\label{thm:characterization N=1}
The following conditions are equivalent for a homomorphism  between Polish groups, $G\overset{\pi}{\longrightarrow}H$. 
\begin{enumerate} 
\item $\pi$ is continuous,
\item $N = \{1\}$,
\item for all identity neighbourhoods $U \subseteq G$ and $V \subseteq H$, there exists a finite set $F \subseteq U$ such that 
$$
\bigcup_{f\in F} f\pi^{-1}(V)  f^{-1}
$$ 
is an identity neighbourhood in $G$.
\end{enumerate}
\end{thm}

\begin{proof}
The equivalence of (1) and (2) has already been shown and that (1) implies (3) is obvious. So assume that (3) holds and assume, without loss of generality, that $\pi[G]$ is dense in $H$. Let $\tilde{\pi}$ be the induced mapping $G\overset{\tilde{\pi}}{\longrightarrow}H/N$, which is then continuous. To see that $N=\{1\}$, by Lemma \ref{lemme:trivialNbd}, we only need to check that, for every neighbourhood $V$ in $H$, there is a finite set $F \subseteq N$ so that $N \subseteq \bigcup_{f \in F} fVf^{-1}$. So, let $V$ be any identity neighbourhood in $H$ and let $W$ be an open identity neighbourhood so that $\overline{WWW^{-1}} \subseteq V$. By the open mapping theorem, $N W$ is open in $H/N$, whereby, as $\tilde{\pi}$ is continuous, $\pi^{-1}(N W)$ is an identity neighbourhood in $G$. By (3), there is a finite set $E \subseteq \pi^{-1}(N W)$ so that
    \begin{align*}
        U = \bigcup_{g \in E} g\pi^{-1}(W)g^{-1}
    \end{align*}
    is an identity neighbourhood in $G$ and by definition $N \subseteq \overline{\pi[U]}$. Now, let $F \subseteq N$ be a finite set so that $\pi[E] \subseteq FW$. Then
    \begin{align*}
        \pi[U] \subseteq \bigcup_{g\in E}\pi(g)W\pi(g)\inv \subseteq \bigcup_{f \in F}fWWW^{-1}f^{-1},
    \end{align*}
    and so
    \begin{align*}
        N \subseteq \overline{\pi[U]} \subseteq \bigcup_{f \in F} f\overline{WWW^{-1}}f^{-1} \subseteq \bigcup_{f \in F}fVf^{-1}
    \end{align*}
    as we wanted. 
\end{proof}

\begin{thm}\cite{rosendal-Pi}\label{thm:characterization N compact}
The following conditions are equivalent for a homomorphism  between Polish groups, $G\overset{\pi}{\longrightarrow}H$. 
\begin{enumerate} 
\item $N$ is compact,
\item for all identity neighbourhoods $U \subseteq G$ and $V \subseteq H$, there exists a finite set $F \subseteq U$ so that 
$$
F\cdot \pi\inv(V)\cdot F
$$
is an identity neighbourhood in $G$,
\item for all identity neighbourhoods $U \subseteq G$ and $V \subseteq H$, there exists a finite set $F \subseteq U$ so that 
$$
F\cdot \pi\inv(V)
$$
is an identity neighbourhood in $G$.
\end{enumerate}
\end{thm}

\begin{proof}Without loss of generality, we may assume that $\pi[G]$ is dense in $H$, whereby $N$ is the smallest closed normal subgroup of $H$ so that $G\overset{{\tilde\pi}}{\longrightarrow}H/N$ is continuous.

(1)$\saa$(3):  Suppose $N$ is compact and let $U \subseteq G$ and $V \subseteq H$ be the given identity neighbourhoods. Fix some symmetric open identity neighbourhood $W \subseteq H$ so that $W^3 \subseteq V$. As $N$ is compact, let $E \subseteq N$ be a finite set so that $N \subseteq EW$. Then, as $E \subseteq N \subseteq \overline{\pi[U]} \subseteq \pi[U]W$, we can find a finite set $F \subseteq U$ with $E \subseteq \pi[F]W$ and thus
\begin{align*}
            NW \subseteq EW^2 \subseteq \pi[F]W^3 \subseteq \pi[F]V.
\end{align*}
Since $G\overset{{\tilde\pi}}{\longrightarrow}H/N$ is continuous, it follows that $\pi^{-1}(NW)$ is open in $G$ and thus also that $\pi^{-1}(\pi[F]V)=F\pi^{-1}(V)$ is an identity neighbourhood in $G$.

The implication (3)$\saa$(2) is immediate.

(2)$\saa$(1):
By Theorem \ref{thm:compactness} it is enough to show that, for every identity neighbourhood $V$ in $H$, there exists a finite set $F \subseteq N$ such that $N \subseteq FVF$. So, let $V$ be given and find a symmetric open neighbourhood $W$ so that $\overline{W^3} \subseteq V$. As the induced mapping $G\overset{\tilde{\pi}}{\longrightarrow}H/N$ is continuous, the set $\pi^{-1}(NW)$ is open in $G$. Therefore, there is a finite set $E \subseteq \pi^{-1}(NW)$ such that $U=E\pi^{-1}(W)E$ is an identity neighbourhood in $G$. As $E \subseteq \pi^{-1}(NW)=\pi^{-1}(WN)$, we have $\pi[E]\subseteq FW\cap WF$ for some finite set $F \subseteq N$. Therefore, by the definition of $N$,
        \begin{align*}
            N \subseteq \overline{\pi[U]} \subseteq \overline{\pi[E]W\pi[E]}\subseteq \overline{FWWWF} \subseteq F\overline{WWW}F \subseteq FVF,
        \end{align*}
proving that $N$ is indeed compact.
\end{proof}

\section{Christensen's Problem}\label{sec:christensen}
As witnessed by Theorems \ref{thm:characterization N=1} and \ref{thm:characterization N compact}, the degree of continuity of $G\overset{{\pi}}{\longrightarrow}H$ depends on the size of $N$. For example,  $\pi$ is continuous if and only if $N=\{1\}$. However, when $G$ is locally compact, it carries a left-invariant Haar measure $\mu $ that also allows us to add considerations of measurability. We begin with a classical result due to H. Steinhaus \cite{steinhaus} for $G=\R$ and A. Weil \cite{weil} for arbitrary $G$.

\begin{thm}\cite{steinhaus,weil}\label{weil}
Let $G$ be a locally compact Polish group with (left) Haar measure $\mu$. Then, for any $\mu$-measurable set of positive measure, $A \subseteq G$, the product $AA^{-1}$ is an identity neighbourhood. 
\end{thm}
\begin{proof}
    By regularity of $\mu$, we may find a compact set $K$ and an open set $U$ such that $K \subseteq A \subseteq U$ and $\mu(U)<2\mu(K)$. Now, pick an open identity neighbourhood $V$ such that $VK \subseteq U$. Then, if $g \in V$, $gK$ is a subset of $U$ of measure $\mu(K)$ and so $K \cap gK \neq \emptyset$. Otherwise, because $gK \cup K \subseteq U$ and they are disjoint, we would have that
\begin{align*}
        \mu(U)\geqslant \mu(gK\cup K)=\mu(gK)+\mu(K)=2\mu(K),
    \end{align*}
    contradicting the choice of $K$ and $U$.
We conclude that $g \in KK^{-1} \subseteq AA^{-1}$. So $V \subseteq AA^{-1}$.
\end{proof}

\begin{cor}
 Let $G$  and $H$ be Polish groups with $G$ locally compact. Then every Haar measurable homomorphism $G\overset{{\pi}}{\longrightarrow}H$  is continuous.
\end{cor}

\begin{proof}
As it suffices to prove continuity at the identity, let $U$ be any given identity  neighbourhood in $H$ and choose a smaller open identity neighbourhood $V \subseteq U$ such that $VV^{-1} \subseteq U$. Choose a dense sequence $(h_n)_n$ in $H$, whereby  $H=\bigcup_{n=1}^\infty Vh_n$.  Because $\pi$ is measurable, $G=\bigcup_{n=1}^\infty \pi^{-1}(Vh_n)$ is a covering by measurable sets, so we may pick some  $m \in \N$ such that $\mu(\pi^{-1}(Vh_m))>0$. By Theorem \ref{weil}, the set 
$$
\pi^{-1}(Vh_m)(\pi^{-1}(Vh_m))^{-1}
$$ 
is an identity neighbourhood in $G$ and hence so is  $\pi^{-1}(U)$.
\end{proof}

The previous result can be extended to all Polish groups $G$ provided that the concept of Haar measurability, which only makes sense for locally compact groups, is replaced with an alternative measure theoretical concept. For this, a subset $A$ of a Polish group $G$ is called \emph{universally measurable} if it is measurable with respect to \emph{every} Borel probability measure or, equivalently, to every $\sigma$-finite Borel measure on $G$. In particular, if $G$ is locally compact Polish, then a universally measurable set is a fortiori Haar measurable. 

In the context of his study of notions of measure theoretically small sets in infinite-dimensional Banach spaces and topological groups, J.P.R. Christensen investigated the question of whether universally measurable homomorphisms between Polish groups are continuous \cite{christensen}. In particular, Christensen proved this to be the case when either the domain or range is abelian. The proof of this result relies on a more precise statement with wider applicability.

If $k\geqslant 2$, we let $C_k =\Z/k\Z$ be the cyclic group of order $k$, which we may of course identify with the finite set $\{0,1,\ldots, k-1\}$. Let also $\prod_{n=1}^\infty C_k $ be the infinite direct product, which is a compact topological group, and let $\bigoplus_{n=1}^\infty C_k $ be the dense subgroup of $\prod_{n=1}^\infty C_k $ consisting of finitely supported elements, that is, elements of the form $(z_1,z_2,\ldots, z_m, 0,0,\ldots)$ with $z_i\in C_k $. As $\prod_{n=1}^\infty C_k $ is abelian, we write the group operation additively.

\begin{thm}\cite{christensen}\label{jpr}
Suppose that $G=\bigcup_{j=1}^\infty A_j$ is a  covering of a Polish group $G$ by universally measurable sets $A_j$ and that $U$ is an identity neighbourhood in $G$. Then there are a finite set $F\subseteq U$ and some $j$ so that 
$$
\bigcup_{g\in F}gA_jA_j\inv g\inv 
$$ 
is an identity neighbourhood.
\end{thm}

\begin{proof}
Suppose for a contradiction that the conclusion fails for $A_1, A_2, \ldots$ and $U$. That is,  for all $h_1,\ldots,h_n\in U$, any $j$ and any identity neighbourhood $V$, 
$$
V\setminus \Big(h_1\inv A_jA_j\inv h_1\cup\ldots\cup h_n\inv A_jA_j\inv h_n\Big)\neq \tom.
$$
Let $\iota\colon \N\to \N$ be a surjection with infinite fibres.
Then we can inductively choose $g_1,g_2,\ldots\in U$ such that for all $i_1<i_2<i_3<\ldots$
\begin{enumerate}
\item any finite product $g_{i_1}g_{i_2}g_{i_3}\cdots g_{i_m}$ belongs to $U$,
\item the infinite product $g_{i_1}g_{i_2}g_{i_3}\cdots$ converges fast to an element of $G$,
\item if $i_1<\ldots <i_m<n$, then 
$$
g_{n}\notin (g_{i_1}\cdots g_{i_{m}})\inv A_{\iota(n)}A_{\iota(n)}\inv(g_{i_1}\cdots g_{i_{m}}).
$$
\end{enumerate}
Here, the convergence in (2) is assumed to be fast enough that the function 
$$
\prod_{n=1}^\infty C_2\maps \phi G
$$ 
defined by setting
$$
\phi(z)=g_1^{z_1}g_2^{z_2}g_3^{z_3}\cdots,
$$
where $g^0=1$ and $g^1=g$, will be continuous.

Note now that some $\phi\inv(A_j)$ must have positive Haar measure in  $\prod_{n=1}^\infty C_2$ and hence, by the Steinhaus--Weil theorem,
$$
\phi\inv(A_j)-\phi\inv(A_j)
$$
is a neighbourhood of the identity ${\bf 0}=(0,0,\ldots)$ in $\prod_{n=1}^\infty C_2$. As the fibre $\iota\inv(j)$ is  infinite, this means that there is some $n\in \iota\inv(j)$ so that
$$
\underbrace{(0,\ldots,0,1,0,0,\ldots)}_{\textrm{$1$ in the $n$th position}}\in \phi\inv(A_j)-\phi\inv(A_j).
$$
In particular, there are two elements $z,u\in \phi\inv(A_j)$ differing exactly in the $n$th coordinate, say $z_n=1$ and $u_n=0$. Thus
$$
\phi(z)=hg_nk \in A_j, \quad \phi(u)=hk\in A_j,
$$
where $h=g_{i_1}g_{i_2}\cdots g_{i_{m}}$, $i_1<\ldots<i_{m}<n$, and $k\in G$.
It follows that
$$
hg_nh\inv=hg_nk\cdot (hk)\inv\in A_jA_j\inv=A_{\iota(n)}A_{\iota(n)}\inv
$$
and so
$$
g_n\in h\inv A_{\iota(n)}A_{\iota(n)}\inv h= (g_{i_1}g_{i_2}\cdots g_{i_{m}})\inv A_{\iota(n)}A_{\iota(n)}\inv(g_{i_1}g_{i_2}\cdots g_{i_{m}}),
$$
contradicting the choice of $g_n$.
\end{proof}

With this result, now it is possible to prove the following general result for Polish groups.

\begin{thm}\cite{rosendal-Pi}\label{thm:universally_measurable_continuous}
Let $G\overset{{\pi}}{\longrightarrow}H$ be a universally measurable homomorphism between Polish groups. Then $\pi$ is continuous.
\end{thm}

\begin{proof}
This follows from Theorems  \ref{thm:characterization N=1} and \ref{jpr}. Indeed, let $V$ be any identity neighbourhood in $H$ and take an identity neighbourhood $W$ such that $WW^{-1} \subseteq V$. Then the set $A=\pi^{-1}(W)$ is universally measurable and, as $H$ is separable, covers $G$ by countably many right translates. By Theorem \ref{jpr}, this means that for any identity neighbourhood $U$ in $G$ there are $g_1, \dots, g_n \in U$ so that 
$$
1\in {\sf int}\Big(\bigcup_{i=1}^n g_i AA^{-1}g_i\inv\Big)\subseteq  {\sf int} \Big( \bigcup_{i=1}^n g_i \pi^{-1}(V)g_i^{-1}\Big).
$$ 
By Theorem \ref{thm:characterization N=1}, we conclude that  $\pi$ is continuous.
\end{proof}

Regarding the problem of automatic continuity of universally measurable homo\-morphisms, T. Banakh \cite[Theorem 1.2]{banakh2022automatic} extended Theorem \ref{thm:universally_measurable_continuous} to the situation when $G$ is \v Cech-complete and $H$ is an aribitrary topological group. To understand this statement, let us note that a topological space $X$ is {\em \v Cech-complete} provided that it is homeomorphic to a $G_\delta$-subset of some compact Hausdorff space. As every Polish space admits a metrisable compactification in which it is necessarily $G_\delta$, \v Cech-complete spaces is a wider class than the Polish spaces. 

\begin{thm}\cite{banakh2022automatic}\label{thm:banakh2022automatic}
Let $G$ be a \v Cech-complete topological group, $H$ be an arbitrary topological group and $G\overset{{\pi}}{\longrightarrow}H$ a universally measurable homomorphism. Then $\pi$ is continuous.
\end{thm}

Whereas we will not concern ourselves with the extension to \v Cech-complete groups, we will show how to extend Theorem \ref{thm:universally_measurable_continuous} to arbitrary topological groups $H$. 
The proof of this, due to Banakh \cite{banakh2022automatic},  is based on  a result of J. Brzuchowski,  J. Cicho\'n, E. Jacek and C. Ryll-Nardzewski \cite{NonMeasurable} generalising an earlier result of L. Bukovsk\'y  \cite[Theorem A]{Bukovsky}.

\begin{thm}\cite{Bukovsky,NonMeasurable}\label{thm:munull}
Let $\mu$ be a Borel probability measure on a Polish space $X$. If $(A_i)_{i\in I}$ is a point-finite family of $\mu$-null sets such that $\bigcup_{i\in J}A_i$ is $\mu$-measurable whenever $J\subseteq I$, then $\bigcup_{i\in I}A_i$ is $\mu$-null.
\end{thm}
Recall that the family $(A_i)_{i\in I}$ is {\em point-finite} provided that every $x\in X$ belongs to only finitely many of the $A_i$. 
In order to apply this in our setting, let us recall the following concept due to Christensen. 

\begin{defi}
Let $G$ be a Polish group and $A \subseteq G$ be a universally measurable set. We say that $A$ is \emph{left Haar null} if there exists a Borel probability measure $\mu$ on $G$ such that $\mu(gA)=0$ for all $g \in G$.
\end{defi}

\begin{lemme}\label{lemma:nonnull}
Let $G$ be a Polish group, $H$  an arbitrary topological group and $G \maps\pi H$ a universally measurable homomorphism. Then, for all identity neighborhoods $U \subseteq H$, the inverse image 
$\pi^{-1}(U)$ is not left Haar null in $G$.
\end{lemme}

\begin{proof}Assume, without loss of generality, that $\pi$ is surjective and suppose for a contradiction that there exists an identity neighborhood $U\subseteq H$ such that $\pi^{-1}(U)$ is left Haar null. Then there exists a Borel probability measure $\mu$ on $G$ so that 
$$
\mu\big(g\cdot \pi^{-1}(U)\big)=0
$$
for all $g \in G$. 

By the Birkhoff--Kakutani metrisation theorem, we may find a continuous left-invariant pseudometric $d$ on $H$ such that $U$ contains the open ball
$$
B=\{g\in G\del d(g,1)<1\}.
$$
So, by the paracompactness of the metric space quotient $H/d$, there exists a locally finite cover $\ku{V}$ of $H$ of $d$-open sets refining the open cover $(hB)_{h \in H}$. In particular, $\ku V$ is a point-finite family and hence so is the family
$$
\ku W=\{\pi\inv (V)\del V\in \ku V\}.
$$
Note also that, for every $\pi\inv(V)\in \ku W$, there is some $h=\pi(g)\in H$ so that 
$$
\pi\inv (V)\subseteq \pi\inv (hB)\subseteq \pi\inv (hU)=g\cdot \pi\inv (U),
$$
implying that $\pi\inv (V)$ is $\mu$-null. Now, as 
$$
G=\bigcup\ku W
$$
has $\mu$-measure $1$, Theorem \ref{thm:munull} implies that there is some subfamily $\ku V'\subseteq \ku V$ so that
$$
\bigcup\big\{\pi\inv(V) \del V\in \ku  V'\big\}=\pi\inv \big(\bigcup \ku V'\big)
$$
is non $\mu$-measurable. However, since $\bigcup \ku V'$ is open in $H$, this contradicticts the universal measurability of $\pi$.
\end{proof}

\begin{defi}
A topological group $H$ is said to be $\omega$\emph{-narrow} if, for every identity neighbourhood $U$, there exist a countable set $C \subseteq H$ so that $H = CU$.
\end{defi}
Let us also observe that, by an argument as in the implication (3)$\saa$(2) in Theorem \ref{thm:compactness}, for $H$ to be $\omega$-narrow, it suffices that, for all identity neighbourhoods $U$, there is a countable set $C$ and a finite set $F$ so that $H=CUF$.

The class of $\omega$-narrow groups was introduced by I.I. Guran \cite{Guran}, who also established the fundamental characterisation of these groups as being the topological subgroups of direct products of Polish groups. 

\begin{lemme}\label{lemme:omega-narrow}
Let $G$ be a Polish group, $H$ a topological group and $G\maps\pi H$ a universally measurable epimorphism. Then $H$ is $\omega$-narrow.
\end{lemme}

 \begin{proof}
Assume towards a contradiction that $H$ is not $\omega$-narrow and find an identity neighbourhood $U\subseteq H$  so that, for all countable sets $C$ and finite sets $F$,
$$
H \neq CUF.
$$ 
Also, take $V$ neighborhood of $1_G$ such that $V^{-1}V \subseteq U$. By Lemma \ref{lemma:nonnull}, the universally measurable set
$$
A=\pi^{-1}(V)
$$ 
is not left Haar null in $G$. So by \cite[Theorem 2.8]{rosendal-BSL} there are $g_{1}, \dots, g_{m}\in G$ so that
$$
W=\bigcup_{i=1}^m g_{i}^{-1}A\inv Ag_{i}
$$
is an identity neighborhood in $G$. In particular, as $G$ is separable, there are $f_1,f_2,\ldots\in G$ so that $G=\bigcup _jf_jW$, whereby
$$
H=\pi[G]=\bigcup_j\pi(f_j)\pi[W]\subseteq\bigcup_j\bigcup_i\pi(f_jg_i\inv)U\pi(g_i),
$$
contradicting the choice of $U$.
 \end{proof}

The proof of Theorem \ref{thm:banakh2022automatic} for the case when the domain group $G$ is Polish is now straightforward. Indeed, suppose $G\maps \pi H$ is a universally measurable homomorphism from a Polish group $G$ into a topological group $H$. Without loss of generality, $\pi$ is surjective, whence, by Lemma \ref{lemme:omega-narrow}, $H$ is $\omega$-narrow. It thus follows from the result of  I.I. Guran \cite{Guran} that $H$ embeds as a topological subgroup of a product
$$
\prod_{i \in I}H_i
$$
of Polish groups $H_i$. Composing $\pi$ with the coordinate projections, we obtain universally measurable homomorphisms $G\maps {\pi_i}H_i$, which are continuous by Theorem \ref{thm:universally_measurable_continuous}. As a function into a product of topological spaces is continuous if and only if each composition with a coordiate projection is continuous, we conclude that $\pi$ is itself continuous.

%%%%%%%%%%%%%%%%%%%%%%%%%%%%%%%%%%%%%%%%%%%%%%%%%%%%
%%%%%%%%%%%%%%%%%%%%%%%%%%%%%%%%%%%%%%%%%%%%%%%%%%%%
%%%%%%%%%%%%%%%%%%%%%%%%%%%%%%%%%%%%%%%%%%%%%%%%%%%%
%%%%%%%%%%%%%%%%%%%%%%%%%%%%%%%%%%%%%%%%%%%%%%%%%%%%
%%%%%%%%%%%%%%%%%%%%%%%%%%%%%%%%%%%%%%%%%%%%%%%%%%%%
%%%%%%%%%%%%%%%%%%%%%%%%%%%%%%%%%%%%%%%%%%%%%%%%%%%%
%%%%%%%%%%%%%%%%%%%%%%%%%%%%%%%%%%%%%%%%%%%%%%%%%%%%
%%%%%%%%%%%%%%%%%%%%%%%%%%%%%%%%%%%%%%%%%%%%%%%%%%%%
%%%%%%%%%%%%%%%%%%%%%%%%%%%%%%%%%%%%%%%%%%%%%%%%%%%%
%%%%%%%%%%%%%%%%%%%%%%%%%%%%%%%%%%%%%%%%%%%%%%%%%%%%
%%%%%%%%%%%%%%%%%%%%%%%%%%%%%%%%%%%%%%%%%%%%%%%%%%%%

\section{A quadrichotomy for homomorphisms}\label{sec:quadri}
Theorems \ref{thm:characterization N=1} and \ref{thm:characterization N compact} provide exact criteria for when the characteristic group $N$ is either trivial or compact. As it turns out, the analysis of $N$ also serves to elucidate the question of how much choice is needed to produce discontinuous homomorphisms between Polish groups. 

It is well-known that some form of the axiom of choice is needed to produce discontinuous homomorphisms between Polish groups. For example, in R.M. Solovay's model \cite{solovay} in which every set of reals is Lebesgue mesurable, every subset of a Polish space also has the property of Baire and thus all homomorphisms between Polish groups are Baire measurable. By simple argument using Pettis' lemma, we then find that all homomorphisms between Polish groups are continuous. Whereas, Solovay's original argument needed an inaccessible cardinal number, it was shown by S. Shelah \cite{shelah} that, in order to get a model of set theory in which every subset of a Polish space has the property of Baire, the inaccessible cardinal can be avoided. 

In order to discuss and compare consequences of the axiom of choice, we must of course work in some suitably weak background theory, which we here take to be ZF+DC. Here DC is the {\em axiom of dependent choice} stated below.
\begin{defi}[Axiom of dependent choice] 
Let  $X$ be a nonempty set and $R$ a  binary relation on $X$ such that, for every $x\in X$, there is some $y\in X$ with $xRy$. 
Then there exists a sequence $(x_n)_{n=1}^\infty$ such that $x_nRx_{n+1}$ for all
$n\in \N$. 
\end{defi}
Using ZF+DC as a background theory is appropriate, as this suffices to establish the basic concepts and results of analysis that do not directly involve choice. Thus, for example, with ZF+DC we can prove the Baire category theorem and perform most arguments involving convergence of sequences, but we cannot prove the existence of non-principal ultrafilters on $\N$ nor establish the existence of non-measurable sets of reals.

%%%%%%%%%%%%

\subsection{Vitali sets and chromatic numbers of Hamming graphs} 
For a fixed number $k\geqslant 2$, the \emph{Hamming graph} on the abelian group $\prod_{n=1}^\infty C_k $ is the graph with vertex set $\prod_{n=1}^\infty C_k $ so that two vertices $z,w \in \prod_{n=1}^\infty C_k $ form an edge if and only if they differ exactly one coordinate $n \in \N$, i.e., if $z-w$ belongs to one of the factor subgroups $C_k $ in the product $\prod_{n=1}^\infty C_k $. 

A \emph{colouring} of the Hamming graph is just a function $c\colon \prod_{n=1}^\infty C_k  \rightarrow X$ into some set $X$ so that $c(z) \neq c(w)$ whenever $\{z,w\}$ is an edge. The \emph{chromatic number} of the Hamming graph on $\prod_{n=1}^\infty C_k $, which we will denote by
$$
\chi(k),
$$
is then the smallest cardinality $\kappa$ so that there is a graph colouring $c\colon \prod_{n=1}^\infty C_k  \rightarrow X$ into a set $X$ of cardinality $\kappa$. Observe that, since each factor subgroup $C_k $ is a clique of size $k$, we have that $\chi(k)\geqslant k$ for all $k$.

Now, suppose $c \colon \prod_{n=1}^\infty C_k  \rightarrow X$ is some graph colouring and that $m=k^n$ for some $n \geqslant 1$. Fix a bijection $\phi : C_m \rightarrow \prod_{i=1}^n C_k $ and let $(\cdot)_i =  \text{proj}_i \circ \phi$ for $i=1, \dots, n$. We then define a function $c_i\colon \prod_{n=1}^\infty C_m \rightarrow X$ by letting
$$
c_i(z)=c((z_1)_i,(z_2)_i,(z_3)_i,\dots)
$$
and let $C \colon \prod_{n=1}^\infty C_m \rightarrow X^n$ be given by $C(z)=(c_1(z),\dots, c_n(z))$. Then $C$ is also a graph colouring, which shows that
$$
\chi(k^n)\leqslant \chi(k)^n
$$
for all $n \geqslant 1$. In particular, this shows that, if $\chi(k)=k$ for some $k$, then actually $\chi(k)=k$ for infinitely many $k$. Observe also that, if $l\leqslant k$, then the Hamming graph on $\prod_{n=1}^\infty C_l$ can be seen as an induced subgraph of that on $\prod_{n=1}^\infty C_k $ , whence $\chi(l)\leqslant \chi (k)$. Since $\chi(k^n)\leqslant \chi(k)^n$, we thus see that $\chi(k)$ is either infinite for all $k$ or finite for all $k$. 

Now, consider the coset equivalence relation  on $\prod_{n=1}^\infty C_k $, denoted $\approx_k$, induced by the subgroup $\bigoplus_{n=1}^\infty C_k $. A \emph{transversal} for $\approx_k$ is a set $T \subseteq \prod_{n=1}^\infty C_k $ containing exactly one representative of each equivalence class. In case $T$ is such a transversal, it can be seen that $\chi(k)=k$. Indeed, a colouring $c: \prod_{n=1}^\infty C_k  \rightarrow C_k $ is defined by letting
$$
c(z)=\sum_{i=1}^\infty(z-\hat{z})_i
$$
where $\hat{z} \in T$ is the unique representative of the equivalence class of $z$ in $T$. Here we note that, because $z-\hat z\in \bigoplus_{n=1}^\infty C_k $, the previous sum is actually  finite.

Similarly, if $\ku U$ is a non-principal ultrafilter on $\N$, then the formula
$$
c(z)=\lim_{n\to \ku U}\;\sum_{i=1}^n z_i.
$$
defines a discontinuous group homomorphism $c: \prod_{n=1}^\infty C_k  \rightarrow C_k $ that is simultaneously a graph colouring. Therefore, if there is a non-principal ultrafilter on $\N$, then $\chi(k)=k$ for all $k\geqslant 2$.

A \emph{Vitali set} is a set $T \subseteq \R$ intersecting every translate of $\Q$ in a single point, that is,  a Vitali set is a set of representatives of the coset equivalence relation induced by $\Q$ as a subgroup of $\R$. For each $k \geqslant 2$, since both $\approx_k$ and the equivalence relation on $\R$ of belonging to the same translate of $\Q$ are hyperfinite equivalence relations, they are Borel bireducible and thus, in every model of ZF+DC, there is a Vitali set if and only if $\approx_k$ admits a transversal. 

Finally, let us observe that, if $c\colon \prod_{n=1}^\infty C_k  \rightarrow X$  is a graph colouring with some countable palette $X$, then some colour class $A=c\inv(x)$ must be non-meagre. By Pettis' lemma, we have $U(A)-U(A)\subseteq A-A$ and hence, if $A$ has the property of Baire, we see that $U(A)\neq \tom$ and thus $A-A$ is an identity neighbourhood in $ \prod_{n=1}^\infty C_k$. From this it would follow that there are two elements of $A$ differing in a single coordinate, contradicting that $c$ is a graph colouring. In other words, if $\chi(k)$ is countable for any $k\geqslant 2$, then there is a subset of the Polish space $ \prod_{n=1}^\infty C_k$ that fails to have the property of Baire. This means that we now have a sequence of implications holding under ZF+DC.
\begin{align*}
    \text{There is a Vitali}&\text{ set or a non-principal ultrafilter on $\N$}\\
    & \Rightarrow \chi(k)=k \text{ for all }k \\ 
    & \Rightarrow \chi(k)=k \text{ for some or, equivalently, infinitely many }k \\ 
    & \Rightarrow \chi(k)<\infty \text{ for all or, equivalently, some } k\\
    & \Rightarrow \text{there is a subset of a Polish space without the property of Baire}.
\end{align*}

%%%%%%%%%%%%%%%%%%%%%%%%%%%%%%%%%%%%%%%%%%%%%%%%%%%%
%%%%%%%%%%%%%%%%%%%%%%%%%%%%%%%%%%%%%%%%%%%%%%%%%%%%
%%%%%%%%%%%%%%%%%%%%%%%%%%%%%%%%%%%%%%%%%%%%%%%%%%%%

\subsection{The quadrichotomy} For the next lemma, a subset $A$ of a group $G$ is said to be \emph{right $\sigma$-syndetic} provided it covers $G$ by countably many right translates, that is, if  $G = \bigcup_{n=1}^\infty Af_n$ for some sequence $(f_n)_{n=1}^\infty $ in $G$.

\begin{lemme}[$\text{ZF+DC}$]\cite{rosendal-Pi}\label{lemma15} 
    Suppose that there is no Vitali set. Then, for every right $\sigma$-syndetic subset $A$ of a Polish group $G$ and identity neighbourhood $U \subseteq G$, there is a finite set $E \subseteq U$ so that
    \begin{align*}
        EAA^{-1}E
    \end{align*}
    has a nonempty interior.
\end{lemme}

\begin{proof}Assume that $EAA^{-1}E$ has empty interior for every finite set $E \subseteq U$. 
By our hypothesis, $G = \bigcup_{n=1}^\infty Af_n$ for some sequence $(f_n)_{n=1}^\infty$ in $G$. Recursively and following the proof of Theorem \ref{jpr} construct a sequence $(g_n)_{n=1}^\infty$ in $U$ satisfying (1) and (3) in Theorem \ref{jpr},  so that the associated function 
$$
\phi \colon \prod_{n=1}^\infty C_2 \rightarrow G
$$ 
is a continuous injection and thus a homeomorphism with its image, and such that for all $i_1<\dots < i_n<k$ and $j_1<\dots<j_m<k$ we have
$$
g_k \notin g_{i_n}^{-1}\dots g_{i_1}^{-1}\cdot AA^{-1} \cdot g_{j_1}\dots g_{j_m}.
$$

    Assume that $z,w\in \prod_{n=1}^\infty C_2$ are $\approx_2$-equivalent but distinct, say $z_k=1, w_k=0$ and $z_n=w_n$ for all $n>k$. Then we can write
    \begin{align*}
        \phi(z) = g_{i_1}\dots g_{i_n}g_kh, \hspace{5mm} \phi(w)=g_{j_1}\dots g_{j_m}h
    \end{align*}
    for some $i_1<\dots<i_n<k, j_1<\dots<j_m<k$ and $h \in G$. Thus by condition (3)
    \begin{align*}
        \phi(z)\phi(w)^{-1} = g_{i_1}\dots g_{i_n}\cdot g_k \cdot g^{-1}_{j_m}\dots g^{-1}_{j_1} \notin AA^{-1}.
    \end{align*}
In particular, this shows that each set $B_m = \phi^{-1}(Af_m)$ can only intersect an $\approx_2$-equivalence class in a single point and hence is a partial $\approx_2$-transversal.
Since $\prod_{n=1}^\infty C_2 = \bigcup_{m=1}^\infty B_m$, this means that a transversal $T \subseteq \prod_{n=1}^\infty C_2$ for $\approx_2$ can be defined by
    \begin{align*}
        T = \bigcup_{m=1}^\infty \Big(B_m \setminus \left[\bigcup_{l=1}^{m-1}B_l\right]_{\approx_2}\Big).
    \end{align*}
    Thus, if the conclusion of the lemma fails, there is a transversal for $\approx_2$ and hence also a Vitali set.
\end{proof}

\begin{lemme}[$\text{ZF+DC}$]\cite{rosendal-Pi}\label{lemma16}
Suppose $G\overset{{\pi}}{\longrightarrow}H$ is a homomorphism between Polish groups so that $N$ is compact and assume that $V\subseteq H$ is a  symmetric open identity neighbourhood so that $hVh^{-1} =V$ for all $h \in N$. Let $p$ be minimal so that
$$
N\subseteq VF
$$
for some set $F\subseteq N$ with $|F|=p$
and let $k$ be maximal so that there are $h_1,\ldots, h_k\in  N$ and  an identity neighbourhood $W\subseteq H$ for which
\maths{
VWh_i\cap VWh_j=\tom
}
for all $i\neq j$.
Then 
$$
\chi(k)\leqslant p.
$$
\end{lemme}

\begin{proof}
Pick some set $F\subseteq N$ of cardinality $p$ so that $N\subseteq VF$ and find $h_1,\ldots, h_k\in N$ and a symmetric identity neighbourhood $W\subseteq H$ so that
\begin{equation}\label{eq:chrom}
VW^2h_i\cap VW^2h_j=\tom
\end{equation}
for all $i\neq j$. Since $N$ is compact and $VF$ is open, we can by shrinking $W$ assume that $NW\subseteq VF$. 
Because  the induced map $G\to \ov{\pi[G]}/N$ is continuous, $\pi\inv(NW)$ is an identity neighbourhood in $G$. So pick a symmetric open identity neighbourhood $U\subseteq G$ so that $\ov U\subseteq \pi\inv(NW)$.

For every identity neighbourhood $O$ in $G$, we have 
$$
h_1,\ldots, h_k\in N\subseteq \ov{\pi[O]}\subseteq W\pi[O].
$$
This means that we may inductively choose finite sets $E_1,E_2,\ldots\subseteq G$ so that, for all $n$, 
\begin{enumerate}
\item[(i)] $E_1\cdots E_n\subseteq U$,
\item[(ii)] $h_1,\ldots, h_k\in W\pi[E_n]$,
\item[(iii)] $VW\pi(g)\cap VW\pi(f)=\tom$ for all distinct $g,f\in E_n$,
\item[(iv)] the map $\phi\colon \prod_{i=1}^\infty E_i\to \ov U$ given by 
$$
\phi(g_1,g_2, g_3, \ldots)=g_1g_2g_3\ldots
$$
is well-defined, injective and continuous with respect to the product topology on  $\prod_{i=1}^\infty E_i$. 
\end{enumerate}
Observe that by (iii) and the definition of $k$ we have $|E_n|\leqslant k$ for all $n$. Conversely, by Equation \ref{eq:chrom} and (ii), we find that also $|E_n|\geqslant k$ for all $n$.

Now, suppose $\alpha,\beta\in \prod_{i=1}^\infty E_i$ differ in a single coordinate $n$. Then we can write $\phi(\alpha)=ugx$ and $\phi(\beta)=ufx$ for some $u\in U\subseteq \pi\inv(NW)$, $x\in G$ and distinct $g,f\in E_n$, whereby $\pi(\phi(\alpha))= hw\pi(g)\pi(x)$ and $\pi(\phi(\beta))= hw\pi(f)\pi(x)$ for some $h\in N$ and $w\in W$. Thus
\[\begin{split}
V\pi(\phi(\alpha))\cap V\pi(\phi(\beta))
&=V hw\pi(g)\pi(x)\cap Vhw\pi(f)\pi(x)\\
&\subseteq hV W\pi(g)\pi(x)\cap hVW\pi(f)\pi(x)\\
&=h\big[VW\pi(g)\cap VW\pi(f)\big]\pi(x)\\
&=\tom.
\end{split}\]
In particular, $\pi(\phi(\alpha))$ and $\pi(\phi(\beta))$ cannot belong to the same right-translate $Vz$ of $V$ by any $z\in H$. 
Since $\text{im}(\phi)\subseteq\ov U\subseteq \pi\inv(NW)\subseteq \pi\inv(VF)$, it thus follows that the sets
$$
A_z=(\pi\phi)\inv(Vz)
$$
for $z\in F$ cover $\prod_{i=1}^\infty E_i$ by sets so that no two distinct elements of the same $A_z$ can differ in a single coordinate. Identifying $\prod_{i=1}^\infty E_i$ with $\prod_{i=1}^\infty C_k $ in the obvious way, this defines a colouring $c\colon \prod_{i=1}^\infty C_k  \to F$, showing that $\chi(k)\leqslant p$.
\end{proof}

Having established these lemmas, we are now ready to reassemble everything into a single quadrichotomy.

\begin{thm}[$\text{ZF+DC}$]\cite{rosendal-Pi}\label{quadrichotomy} One of the following conditions hold.
\begin{enumerate}
    \item Every homomorphism between Polish groups is continuous.
    \item The chromatic number $\chi(k)$ is finite for all $k\geqslant 2$ and, if $G\overset{{\pi}}{\longrightarrow}H$ is a homomorphism between Polish groups, then $N$ is compact and connected.
    \item For infinitely many $k \geqslant 2$, we have $\chi(k)=k$ and, if $G\overset{{\pi}}{\longrightarrow}H$ is a homomorphism between Polish groups, then $N$ is compact.
    \item There is a Vitali set.
\end{enumerate}
\end{thm}

\begin{proof}
In all of the proof, we suppose that (4) fails, that is, there is no Vitali set. Then by Lemma \ref{lemma15}, for every right $\sigma$-syndetic set $A \subseteq G$ and identity neighbourhood $U \subseteq G$, there is a finite set $E \subseteq U$ so that $EAA^{-1}E$ is an identity neighbourhood in $G$. 

Suppose $G\overset{{\pi}}{\longrightarrow}H$ is a homomorphism between Polish groups and so that $\pi[G]$ is dense in $H$ and let $U \subseteq G$ and $V \subseteq H$ be identity neighbourhoods. Let $W$ be a symmetric open identity neighbourhood such that $W^2 \subseteq V$. Then $A = \pi^{-1}(W)$ is right $\sigma$-syndetic in $G$ and so, for some finite $E \subseteq U$,
$$
E \pi^{-1}(W) \pi^{-1}(W)^{-1}E \subseteq E \pi^{-1}(V)E
$$ 
are identity neighbourhoods in $G$. By Theorem \ref{thm:characterization N compact}, we conclude that $N$ is compact.

Now, assume there is some $G\overset{{\pi}}{\longrightarrow}H$ so that $N$ is compact, but not connected. Then by a theorem of D. van Dantzig \cite{Dantzig} $N$  has a clopen proper normal subgroup $M\trianglelefteq N$. Because $N$ is  compact, the index $k=[N:M]>1$ is finite. Letting $h_1, \ldots, h_k\in N$ be coset representatives for $M$ in $N$, we thus see that
$$
Mh_1, \ldots, Mh_k
$$
is a partition of $N$ into compact sets. We may thus choose a symmetric open identity neighbourhood $W \subseteq H$ such that $W^3Mh_1, \dots, W^3 Mh_k$ are also pairwise disjoint. Moreover, as $N$ is compact, we may suppose that $hWh^{-1} = W$ for all $h \in N$. Then $V=MW$ is symmetric and open in $H$ and so
$$
VW^2h_i\cap VW^2h_j=MW^3h_i\cap MW^3h_j = W^3Mh_i\cap W^3Mh_j=\tom
$$ 
for all $i\neq j$. By Lemma \ref{lemma16}, we find that $\chi(k)=k$ and thus $\chi(k^m)=k^m$ for all $m \geqslant 1$.

Now, assume instead there is a discontinuous homomorphism $G\overset{{\pi}}{\longrightarrow}H$ between Polish groups. Then $N$ is compact but $N\neq \{1\}$. We may thus find some symmetric open identity neighbourhood $V \subseteq H$ so that $hVh^{-1}=V$ for all $h \in N$ and so that $V^3h_1 \cap V^3h_2 = \emptyset$ for some elements $h_1, h_2 \in N$. Letting $W=V$ and  $F \subseteq N$ be any finite set so that $N \subseteq VF$, it follows from Lemma \ref{lemma16}  that $\chi(2) \leqslant |F|$ and hence that $\chi(k)<\infty$ for all $k \geqslant 2$.
\end{proof}

It is worth pausing to consider the ramifications of Theorem \ref{quadrichotomy}. The proper reading of the theorem is of course that every set theoretical universe verifying ZF+DC must satisfy one of the four alternatives. So let us begin with some results indicating the necessity of these. First of all, in Solovay's model, every homomorphism between Polish groups is continuous and $\chi(2)$ is uncountable, which means that it only satisfies Theorem \ref{quadrichotomy}(1). Similarly, if the axiom of choice holds, then there is a discontinuous homomorphism $\R\to \R$ and a Vitali set. Since $\R$ has no non-trivial compact subgroups, this means that we are in case (4) of Theorem \ref{quadrichotomy} and no other. For further cases, we need the following results from the recent book \cite{larson2} by P. B.  Larson and J. Zapletal that contains a wealth of further information about models of ZF+DC. 
\begin{thm}[$\text{ZF+DC}$]\label{thm:zapletal} Relative to the existence of an inaccessible cardinal, the following options are separately consistent.
\begin{enumerate}
\item There is a discontinuous homomorphism $\R\to \R/\Z$, but no Vitali set \cite[Corollary 9.2.21]{larson2}.
\item There is a non-principal ultrafilter on $\N$, but no Vitali set \cite{diprisco}\;\&\;\cite[Corollary 9.2.5]{larson2}.
\item There is a Vitali set, but no discontinuous homomorphism between Polish groups, \cite[Corollary 13.3.10]{larson2}.
\end{enumerate}
\end{thm}
Thus, Theorem \ref{thm:zapletal}(3) implies that it is possible to realise all cases of Theorem \ref{quadrichotomy} simultaneously. On the other hand, Theorem \ref{thm:zapletal}(2) is only compatible with Theorem \ref{quadrichotomy}(3), since a non-principal ultrafilter induces a discontinuous homomorphism $\prod_{n=1}^\infty C_2\to C_2$ and hence with disconnected $N$. And lastly, Theorem \ref{thm:zapletal}(1) is a priori compatible with both (2) and (3) of Theorem \ref{quadrichotomy}, but not with (1) and (4).

A second aspect to consider is what more detailed information can be obtained about each case of Theorem \ref{quadrichotomy}. For example, when restricting to specific target groups $H$, one can exclude certain of the different options and thus get a stronger statement. The simplest instance of this is  when $H$ is the additive topological group $(X,+)$ of some separable Banach space $(X,\norm\cdot)$. Then the only compact subgroup of $H$ is just the trivial group $\{0\}$, which means that, if $G\maps\pi H$ is a group homomorphism with compact characteristic group $N$, then $N=\{0\}$ and so $\pi$ is continuous. This fact was previously established by Larson and Zapletal \cite{larson}.

\begin{cor}[$\text{ZF+DC}$]\cite{larson} \label{cor1:larson} Suppose there is no Vitali set. Then every homomorphism $G\maps \pi H$ from a Polish group $G$ to the additive group $H$ of a separable Banach space is continuous.
\end{cor}
In particular, if there is no Vitali set, then every linear operator between separable Banach spaces is continuous.  Observe also that, whereas a discontinuous homomorphism $\R\to \R$ is incompatible with cases (1), (2) and (3) of Theorem  \ref{quadrichotomy}, on the face of it, a discontinuous homomorphism $\R\to \R/\Z$ is only incompatible with case (1).

There are other consequences of the axiom of choice that are more directly associated with discontinuous homomorphisms. For example, suppose $B\subseteq \R$ is a Hamel basis for $\R$ over $\Q$ and assume without loss of generality that $1\in B$. Then the functional $\R\maps\phi \Q$ that to $x\in \R$ associates the coefficient of $1$ in the expansion of $x$ is a discontinuous homomorphism and the kernel $T={\sf ker}\,\phi$ is a Vitali set.
 It is thus natural to attempt to separate the existence of a Hamel basis from the existence of a Vitali set via a condition on the size of the characteristic group $N$ along the lines of Theorem \ref{quadrichotomy}.
\begin{prob}[ZF+DC]\label{prob:hamel}
Is it possible to further separate case (4) of Theorem \ref{quadrichotomy} into two options, one stipulating the existence of a Hamel basis for $\R$ over $\Q$ and the other stipulating a non-trivial condition to be satisfied by the characteristic group $N_\pi$ associated with any Polish group homomorphism $G\maps\pi H$?
\end{prob}

%%%%%%%%%%%%%%%%%%%%%%%%%%%%%%%%%%%%%%%%%%%%%%%%%%%%
%%%%%%%%%%%%%%%%%%%%%%%%%%%%%%%%%%%%%%%%%%%%%%%%%%%%
%%%%%%%%%%%%%%%%%%%%%%%%%%%%%%%%%%%%%%%%%%%%%%%%%%%%

\subsection{Small index properties} 
A shall soon be evident, the problem of continuity of homomorphisms is closely tied to other more explicitly algebraic questions for which it will be useful to introduce some specific terminology.
\begin{defi}\label{def:automcont}
Let $G$ be a Polish group. We say that $G$ has the
\begin{enumerate}
\item {\em automatic continuity property} if every homomorphism $G \maps\pi H$ into any Polish group $H$ is continuous,
\item {\em countable index property} if every subgroup $F\leqslant G$ of countable index is open,
\item  {\em normal countable index property} if every normal subgroup $F\leqslant G$ of countable index is open,
\item  {\em finite index property} if every subgroup $F\leqslant G$ of finite index is open.
 \end{enumerate}
 \end{defi}
It is trivial that the countable index property implies the normal countable index property. Also, by Poincaré's lemma, every finite index subgroup $F$ of a (topological) group $G$ contains a finite index normal subgroup $K$ of $G$. Thus, if $K$ is open in $G$, so is $F$. Therefore, the normal countable index property, in turn, implies the finite index property.

Evidently, the finite index property for a group $G$ is equivalent to requiring all homomorphisms to finite discrete groups to be continuous and, similarly, the normal countable index property is equivalent to requiring all homomorphisms to countable discrete groups are continuous. 

On the other hand, the countable index property has two interesting reformulations. Namely, suppose $X$ is a countable discrete set and equip the group $S={\sf Sym}(X)$ of all permutations of $X$ with the {\em permutation group topology} obtained by declaring the isotropy subgroups
$$
S_x=\mgd{f\in S}{fx=x}
$$
to be open. This is the same as the topology of pointwise convergence on the discrete set $X$ and, in this topology, ${\sf Sym}(X)$ becomes a totally disconnected Polish group.
We may then observe that the following three conditions are equivalent for a Polish group $G$.
\begin{enumerate}
\item[(i)] $G$ has the countable index property,
\item[(ii)] every action $G\curvearrowright X$ on a countable discrete set $X$ is continuous, that is, the action map $G\times X\to X$ is continuous,
\item[(iii)] for every countable discrete set $X$, every homomorphism $G\maps\pi {\sf Sym}(X)$ is continuous.
\end{enumerate}
Indeed, continuity of an action $G\curvearrowright X$ simply means that the (countable index) isotropy subgroups 
$$
G_x=\mgd{g\in G}{gx=x}
$$ 
are open in $G$, which shows (i)$\saa$(ii). Also, a homomorphism $G\maps\pi {\sf Sym}(X)$ is continuous if the associated action $G\acts{}X$ is continuous, which establishes (ii)$\saa$(iii). And finally, a  countable index subgroup $F\leqslant G$ is open if  the homomorphism $G\to{\sf Sym}(G/F)$ induced by the left-multiplication action of $G$ on the discrete set $G/F$ is continuous, whence (iii)$\saa$(i).

Observe that because of the equivalence (i)$\equi$(iii) we also see that the automatic continuity property implies the countable index property. We therefore have the following implications among the four properties of Definition \ref{def:automcont}.
\maths{
\text{Automatic continuity property} 
&\quad \saa \quad \text{Countable index property}\\
&\quad \saa \quad \text{Normal countable index property}\\
&\quad \saa \quad \text{Finite index property}.
}

The most prominent example of the finite index property is the result of N. Nikolov and D. Segal \cite{nikolov} stating that every every topologically finitely generated profinite group has the finite index property (or, in the language of profinite groups, is {\em strongly complete}). Also, J. Saxl and J. Wilson \cite{saxl} have shown that, if $(S_n)_{n=1}^\infty$ is a sequence of non-abelian finite simple groups so that $\lim_n|S_n|=\infty$, then the direct product $\prod_{n}S_n$ has the finite index property. On the other hand, S. Thomas \cite{thomas} shows that these products $\prod_{n}S_n$ need not satisfy the countable index property, whence the finite and countable index properties do not agree in the context of Polish profinite groups.

We do not know whether there is any separation between the automatic continuity and the countable index properties for Polish groups. Although this may simply be due to ignorance on our part, it appears that all constructions of discontinuous homomorphisms in fact also give rise to discontinuous actions on countable sets.
\begin{conj}\label{conj:sip-autom}
Every Polish group with the countable index property also has the automatic continuity property.
\end{conj}

As will be discussed further in Section \ref{section:isometry}, there is a considerable connection between the size of conjugacy classes in a Polish groups and automatic continuity phenomena. For an easy instance, suppose $G$ is Polish group with a comeagre conjugacy class $C\subseteq G$. Then, if $F$ is a countable index normal subgroup of $G$, as $F$ covers $G$ by countably many translates, it cannot be meagre in $G$ and must therefore intersect the conjugacy class $C$. By normality, this means that actually $C\subseteq F$, whence $F$ is a comeagre subgroup and therefore cannot have any disjoint cosets. That is, $F=G$. This shows that $G$ has the normal countable index property. We expect more to be true.
\begin{conj}\label{conj:comeagre-sip}
Every Polish group with a comeagre conjugacy class has the countable index property.
\end{conj}
Taken together, Conjectures \ref{conj:sip-autom} and \ref{conj:comeagre-sip} would thus imply that every Polish group $G$ with a comeagre conjugacy class has the automatic continuity property. If the hypothesis on $G$ is strengthened to having ample generics, which will be defined in Section \ref{section:isometry}, this is indeed the case \cite{kec-ros}. 

Just as Theorem \ref{quadrichotomy} was used to prove Corollary \ref{cor1:larson}, we may investigate the relationship between the  continuity and index properties above under the different options of Theorem \ref{quadrichotomy}. Evidently, if every homomorphism between Polish groups is continuous (which is Theorem \ref{quadrichotomy}(1)), then all Polish groups have the automatic contuity property. Similarly, suppose that Theorem \ref{quadrichotomy}(2) holds and $G\maps\pi {\sf Sym}(X)$ is a homomorphism from a  Polish group $G$ into the group of permutations of a countable discrete set $X$, then the associated characteristic group $N$ is a compact connected subgroup of the totally disconnected group ${\sf Sym}(X)$, which means that $N=\{1\}$ and hence $\pi$ is continuous. We thus have the following result.

\begin{cor}[$\text{ZF+DC}$] Either every Polish group has the countable index property or $\chi(k)=k$ for 
infinitely many $k\geqslant 2$. 
\end{cor}

The more interesting conclusions  however are obtained for Theorem \ref{quadrichotomy}(3).

\begin{thm}[$\text{ZF+DC}$]\label{thm1:noVitali}
Suppose there is no Vitali set. Then, for every countable index subgroup $F$ of a Polish group $G$, there is an open subgroup $G_0\leqslant G$ containing $F$ as a finite index subgroup. If furthermore $F$ is normal in $G$, then so is $G_0$.
\end{thm}

\begin{proof}
Suppose there is no Vitali set and that $G\acts{} X$ is an action of a Polish group on a countable discrete set. Let $G\maps\pi{\sf Sym}(X)$ be the corresponding homomorphism, $N$ the associated characteristic group, and set $H=\ov{\pi[G]}$. Because $N$ is normal in $H$, we have that
$$
hNx=Nhx
$$
for all $h\in H$ and $x\in X$. It follows that the quotient group $H/N$ permutes the $N$-orbits on $X$, which induces a continuous action $H/N\acts{}{\sf Sym}(X/N)$. Because the induced homomorphism $G\maps{\tilde \pi} H/N$ is continuous, it follows that the induced action $G\acts{}X/N$ is also continuous. Therefore, for every $x\in X$,  the stabiliser 
$$
G_{Nx}=\mgd{g\in G}{gNx=Nx}
$$
is open in $G$. Also, as $N$ is compact, $Nx$ is a finite set, whereby the isotropy group $G_x$ has finite index in $G_{Nx}$.

If now $F$ is a countable index subgroup of a Polish group $G$, we may set $X=G/F$ and find that $F=G_{1F}$ has finite index in some open subgroup $G_0\leqslant G$. Now, if $F$ is furthermore normal in $G$, then 
$$
F\leqslant \bigcap_{g\in G}gG_0g\inv \normal G.
$$
As $\bigcap_{g\in G}gG_0g\inv $ is the intersection of open and thus clopen subgroups, it is itself closed and, having countable index in $G$, it is open too. Replacing $G_0$ with $\bigcap_{g\in G}gG_0g\inv $, the last statement of the theorem follows.
\end{proof}

It is remarkable that the conclusion of Theorem \ref{thm1:noVitali} closely follows the main results obtained by O. Bogopolski and S.M. Corson \cite{bogopolski}, albeit for specific countable discrete target groups and no assumption regarding the inexistence of Vitali sets. For example, \cite[Theorem C]{bogopolski} states that, if $G\maps\pi {\sf MCG}(\Sigma)$ is a homomorphism from a Polish group to the (countable discrete) mapping class group of a compact connected surface $\Sigma$, then there is an open normal subgroup $G_0\leqslant G$ in which ${\sf ker}\,\pi$ has finite index. 

\begin{cor}[$\text{ZF+DC}$]\label{cor1:noVitali}
Suppose there is no Vitali set. Then every countable index subgroup $F$ of a compact or a connected Polish group $G$ must have finite index.
\end{cor}
For example, Thomas' \cite{thomas} examples of profinite Polish groups with the finite index, but not countable index property require at least a Vitali set for their proofs.

Finally, we are have a conditional positive answer to Conjecture \ref{conj:comeagre-sip}.
\begin{thm}[$\text{ZF+DC}$]\label{thm2:noVitali} Suppose there is no Vitali set. Then every Polish group with a comeagre conjugacy class has  the automatic continuity property.
\end{thm}

\begin{proof}
We shall show the stronger property that, if $G$ is a Polish group so that the union $C$ of non-meagre conjugacy classes is dense in $G$, then $G$ has the automatic continuity property. To see this, suppose $G\maps\pi H$ is a homomorphism
to a Polish group and let $N$ be the associated characteristic subgroup. Assuming, as we do, that there is no Vitali set, $N$ is compact. 

Let $W\subseteq H$ be a given identity neighbourhood. Because $N$ is compact, we may find a smaller  identity neighbourhood $V$ so that
$$
 V=V\inv=nVn\inv\subseteq V^{6}\subseteq W
 $$
 for all $n\in N$. Set also $U=\pi\inv(VN)$, which is open in $G$. Then $\pi\inv(V)$ must be non-meagre in $G$ and therefore intersect the comeagre set $C$. So fix some $a\in \pi\inv(V)\cap C$. Let also $D\subseteq G$ be a countable dense subset and, for $X\subseteq G$, set $a^X=\mgd{xax\inv}{x\in X}$. We then see that
$$
a^G=a^{DU}=\bigcup_{g\in D}a^{gU}=\bigcup_{g\in D}ga^{U}g\inv
$$
and so, as the conjugacy class $a^G$ is non-meagre, so is the set $a^U$. In particular, by Pettis' lemma and the fact that $a^U$ is analytic and therefore has the Baire property, 
$$
O=a^U\cdot \big(a^U\big)\inv
$$
is an identity neighbourhood in $G$. 

Suppose $g\in O$. Then we may find $u_1,u_2\in U$ so that $g=u_1au_1\inv u_2a\inv u_2\inv$. Furthermore, as $\pi(u_i)\in VN$, there are $n_1,n_2\in N$ so that $\pi(u_i)\in Vn_i$. It thus follows that
\maths{
\pi(g)
&=\pi(u_1)\pi(a)\pi(u_1)\inv \pi(u_2)\pi(a)\inv \pi(u_2)\inv\\
&\in Vn_1\cdot V\cdot n_1\inv V\cdot  Vn_2\cdot V \cdot  n_2\inv V\\
&=V^6\\
&\subseteq W.
}
In other words, $\pi\inv(W)$ is an identity neighbourhood in $G$, showing continuity of $\pi$.
\end{proof}

\begin{exa}[ZF+DC]Although we do not have a general method of constructing non-principal ultrafilters on $\N$ from a discontinuous homomorphism, Thomas and Zapletal \cite[Theorem 4.6]{thomaszapletal} are able to do so in special cases. The following is a slight variation on their methods. 

Fix a sequence $(H_n)_{n=1}^\infty$ of non-trivial countable discrete perfect groups with uniformly bounded commutator width. That is, there is some $r\geqslant 1$ such that, for every $n$ and every $h\in H_n$, one can write $h$ as a product of $r$ commutators, 
$$
h=[g_1,f_1]\cdots [g_r,f_r]
$$
with $g_i,f_i\in H_n$. Then one of the following conditions holds.
\begin{enumerate}
    \item $\prod_{n=1}^\infty H_n$ has the normal countable index property,
    \item there is a non-principal ultrafilter $\ku U$ on $\N$.
\end{enumerate}

To see this, assume that  $G$ is a non-open countable index proper normal subgroup of $\prod_{n=1}^\infty H_n$. Then $\ov G$ is a closed countable index subgroup and thus must be open in $\prod_{n=1}^\infty H_n$. It follows that $\prod_{n\in C}H_n\leqslant \ov G$
for some cofinite subset $C\subseteq \N$ and hence $F=G\cap \prod_{n\in C}H_n$ is a non-open dense countable index normal subgroup of $\prod_{n\in C}H_n$.

We say that a coset $tF$  is {\em full} for a subset $A\subseteq C$ provided that, for every $g\in \prod_{n\in A} H_n$, there is some element of $\hat g\in tF$ so that $g(n)=\hat g(n)$ for all $n\in A$. Observe that, in this case, also $F$ is full for $A$. Moreover, suppose $h\in \prod_{n\in A} H_n$ and write
$$
h=[g_1,f_1]\cdots [g_r,f_r]
$$
for elements $g_i,f_i\in\prod_{n\in A} H_n$. Find $\hat g_i\in F$ as above, whence $g_if_ig_i\inv=\hat g_if_i\hat g_i \inv$. Because $F$ is normal in $\prod_{n\in C}H_n$, it then follows that
$$
h
=g_1f_1g_1\inv f_1\inv  \cdots  g_rf_rg_r\inv f_r\inv 
=\hat g_1 f_1\hat g_1\inv  f_1\inv \cdots \hat g_r f_r\hat g_r\inv  f_r\inv \in F.
$$
In other words, a coset $tF$ is full for $A$ if and only if $\prod_{n\in A} H_n\leqslant F$.

We now set
$$
\ku I=\MGD{A\subseteq C}{\prod_{n\in A} H_n\leqslant F}=\Mgd{A\subseteq C}{\text{ some coset of $F$ is full for $A$ } }
$$
and note that $\ku I$ is a proper ideal on $C$. Observe also that, because $F$ is dense in $\prod_{n\in C}H_n$, it is full for every finite subset $A\subseteq C$. So $\ku I$ contains the Fréchet ideal of all finite subsets of $C$.

We claim that there is some set 
$A\notin \ku I$ such that the trace
$$
\ku I\cap \ku P(A)
$$
is a  prime ideal on $A$. If this fails, then for every $A\notin \ku I$ there is some $A'\subseteq A$ so that neither $A'$ nor $A\setminus A'$ belongs to $\ku I$. Using this, we construct a partition
$$
C=\bigsqcup_{m=1}^\infty A_m
$$
with $A_m\notin \ku I$. Fix also a countable set $\{t_m\}_{m=1}^\infty$ of left-coset representatives for $F$ in $\prod_{n\in C} H_n$. Because  $t_mF$ is not full for $A_m$, there is $g_m\in \prod_{n\in A_m} H_n$ that is not the projection of any element of $t_mF$ and so the product $\prod_{m=1}^\infty g_m$ must lie in $\prod_{n\in C} H_n\setminus \bigcup_{m=1}^\infty t_mF$, which is absurd.

So fix  $A\notin \ku I$ so that $\ku I\cap \ku P(A)$
is a  prime ideal on $A$. Since it contains all finite sets of $A$, the complementary ultrafilter on $A$ is non-principal. Since $A$ must be countably infinite, we can transfer this to a non-principal ultrafilter on $\N$.
\end{exa}

\begin{prob}[ZF+DC]\label{prob:ultra} Is there a dichotomy between, on the one hand,  the existence of a non-principal ultrafilter on $\N$ and, on the other hand,  a non-trivial condition to be satisfied by the characteristic group $N_\pi$ associated with any Polish group homomorphism $G\maps\pi H$?
\end{prob}

%%%%%%%%%%%%%%%%%%%%%%%%%%%%%%%%%%%%%%%%%%%%%%%%%%%%
%%%%%%%%%%%%%%%%%%%%%%%%%%%%%%%%%%%%%%%%%%%%%%%%%%%%
%%%%%%%%%%%%%%%%%%%%%%%%%%%%%%%%%%%%%%%%%%%%%%%%%%%%
%%%%%%%%%%%%%%%%%%%%%%%%%%%%%%%%%%%%%%%%%%%%%%%%%%%%
%%%%%%%%%%%%%%%%%%%%%%%%%%%%%%%%%%%%%%%%%%%%%%%%%%%%
%%%%%%%%%%%%%%%%%%%%%%%%%%%%%%%%%%%%%%%%%%%%%%%%%%%%
%%%%%%%%%%%%%%%%%%%%%%%%%%%%%%%%%%%%%%%%%%%%%%%%%%%%
%%%%%%%%%%%%%%%%%%%%%%%%%%%%%%%%%%%%%%%%%%%%%%%%%%%%
%%%%%%%%%%%%%%%%%%%%%%%%%%%%%%%%%%%%%%%%%%%%%%%%%%%%
%%%%%%%%%%%%%%%%%%%%%%%%%%%%%%%%%%%%%%%%%%%%%%%%%%%%
%%%%%%%%%%%%%%%%%%%%%%%%%%%%%%%%%%%%%%%%%%%%%%%%%%%%
%%%%%%%%%%%%%%%%%%%%%%%%%%%%%%%%%%%%%%%%%%%%%%%%%%%%
%%%%%%%%%%%%%%%%%%%%%%%%%%%%%%%%%%%%%%%%%%%%%%%%%%%%
%%%%%%%%%%%%%%%%%%%%%%%%%%%%%%%%%%%%%%%%%%%%%%%%%%%%
%%%%%%%%%%%%%%%%%%%%%%%%%%%%%%%%%%%%%%%%%%%%%%%%%%%%

\section{Automatic continuity in compact manifolds}\label{sec:homeo}
Henceforth, we shall discard all measurability assumptions and instead consider the problem of finding Polish groups $G$ with the automatic continuity property.
Our main tool for this is the following formally stronger and more  explicit combinatorial property. 

\begin{defi}
A Polish group is said to be \emph{Steinhaus} if there is an exponent $k \geqslant 1$ such that whenever $W \subseteq G$ is symmetric right $\sigma$-syndetic subset containing the identity, then
$$
W^k
$$
is an identity neighbourhood. 
\end{defi}

\begin{lemme}\cite{ros-sol} If $G$ is a Steinhaus Polish group, then any homomorphism $G \maps\pi H$ into a separable group is continuous.
\end{lemme}

Our goal is now to shown that homeomorphism groups of compact manifolds are Steinhaus. For this, let $M$ be a compact manifold of dimension $n$ and consider  the group ${\sf Homeo}(M)$ of all homeomorphisms of $M$. Because $M$ is compact metrisable, one may fix a compatible metric $d$ on $M$ and equip 
${\sf Homeo}(M)$ with the corresponding uniform metric,
$$
d_\infty(g,h)=\sup_{x\in M}d\big(g(x), h(x)\big).
$$
We note that $d_\infty$ is a right-invariant metric inducing a Polish group topology on ${\sf Homeo}(M)$.
This topology agrees both with the compact-open topology  and with the topology of uniform convergence on $M$. In particular, it is independent of the specific choice of the metric $d$ on $M$. 

Our proof relies on the foundational work on topological manifolds due to R.D. Edwards and R.C. Kirby \cite{edwards} encapsulated in the followed innocuous looking but immensely powerful result.

\begin{thm}\cite{edwards}\label{thm:Edwards Kirby} Let  $\{U_1, \ldots, U_m\}$ be an open cover of a compact manifold $M$. Then the set of products
$$
\{g_1g_2 \cdots g_m \del {\sf supp}(g_i) \subseteq U_i\}
$$
is an identity neighborhood in  ${\sf Homeo}(M)$.
\end{thm}

In addition to this, we need some basic facts about homeomorphisms and compact manifolds. Also, to keep matters simple, we exclusively cover the case of closed manifolds, i.e., compact manifolds without boundary. The details for the boundary case can be found in \cite{mann}.

\begin{lemme}\label{lemma:comm}
Suppose $g \in {\sf Homeo}(\R^n)$ is compactly supported. Then there are compactly supported $f,h \in {\sf Homeo}(\R^n)$ so that 
$$
g = f\inv h\inv fh.
$$
\end{lemme}

\begin{proof}
Let $U_0 \subseteq \R^n$ be a bounded open ball containing the support of $g$ and extend this to a bi-infinite  sequence $(U_n)_{n\in \Z}$ of open balls  $U_n\subseteq \R^n$ so that $\bigcup_{n\in \Z}U_n$ is bounded and 
$$
\overline{U_n}\cap \overline{\bigcup_{i\neq n}U_i}=\emptyset
$$
for all $n$. We can now find a compactly supported homeomorphism $h \in {\sf Homeo}(\R^n)$ such that  $h[U_n]=U_{n+1}$ for all $n\in \Z$ and define $f$ by letting 
$$
f\upharpoonright_{U_n} = h^ngh^{-n}\upharpoonright_{U_n}
$$ 
for $n \geqslant 1$ and setting $f$ to be the identity function everywhere else. Then, $h\inv fh$ is the identity outside of $\bigcup_{n=0}^\infty U_n$, 
$$
h\inv fh\upharpoonright_{U_0}= g\upharpoonright_{U_0},
$$
whereas for $n \geqslant 1$
$$
h\inv fh\upharpoonright_{U_n} 
= h\inv (h^{n+1}gh^{-n-1})h\upharpoonright_{U_n} = h^ngh^{-n}\upharpoonright_{U_n}=f\upharpoonright_{U_n}.
$$
Therefore,
$$
f\inv \cdot h\inv fh=g
$$
as claimed.
\end{proof}

For a proof of the following, see  \cite[Lemma 3.10]{mann}.
\begin{lemme}\label{lemma:coveringballs}
Let $M$ be a closed manifold with a compatible metric $d$. Then there is $m = m(M)$ such that, for all $\varepsilon>0$, there exist an open cover $\{U_1,\ldots, U_m\}$ of $M$ where each $U_i$ is a union $U_i = A_{i,1}\cup\cdots\cup A_{i,p}$ of disjoint  open balls $A_{i,j}$ of diameter $<\eps/2$.
\end{lemme}
Let us note that by {\em open ball} we understand the image $A=\phi[\R^r]$ of $\R^r$ under a homeomorphic embedding $\R^r\maps \phi M$, where $r$ is the dimension of $M$. In particular, this will in general not be a ball with respect to the metric $d$.

We aim to show that ${\sf Homeo}(M)$ is $20m$-Steinhaus, where $m=m(M)$ is the constant from Lemma \ref{lemma:coveringballs}. For this purpose, let $W \subseteq {\sf Homeo}(M)$ be a given right $\sigma$-syndetic symmetric set containing the identity. Fix also a sequence $(k_n)_{n=1}^\infty$ in 
${\sf Homeo}(M)$ such that ${\sf Homeo}(M) = \bigcup_{n=1}^\infty Wk_n\inv =\bigcup_{n=1}^\infty k_n W$.
Then, by the Baire category theorem, some $k_nW$ is somewhere dense, whereby $W^2=(k_nW)\inv k_nW$ is dense in an identity neighbourhood, which we may take to be of the form
$$
V_\eps=\{g\in {\sf Homeo}(M)\del d_\infty(g,{\sf id})<\eps\}
$$
for some $\eps>0$.

\begin{lemme}\label{lemma:refinement}
For all open, non-empty $A_1, \dots, A_p \subseteq M$ there are open, nonempty $B_i \subseteq A_i$ such that 
$$
\big\{ h \in {\sf Homeo}(M)\del {\sf supp}(h) \subseteq \bigcup_{i=1}^p B_i\big\}   \subseteq W^{16}.
$$
\end{lemme}

\begin{proof}
Suppose for simplicity that $p=1$. Otherwise the argument below has to be done for the $p$ sets in parallel. Without loss of generality $A_1$ has diameter $<\eps$.  Pick a sequence $(C_n)_{n=1}^\infty$ of non-empty open sets $C_n\subseteq A_1$ satisfying
$$
\overline{C_n}\cap \overline{\bigcup_{i\neq n}C_i} = \emptyset
$$
for all $n$. This later condition ensures that, if $g_n \in {\sf Homeo}(M)$ are such that ${\sf supp}(g_n) \subseteq C_n$, then there is some $g \in {\sf Homeo}(M)$ with ${\sf supp}(g) \subseteq \bigcup_{m=1}^\infty  C_m$ and 
$$
g\upharpoonright_{C_n} = g_n \upharpoonright_{C_n}
$$
for all $n$.

Suppose for a contradiction that, for all $n$, we can find some $g_n \in {\sf Homeo}(M)$ satisfying ${\sf supp}(g_n) \subseteq C_n$, but such that for all $f \in k_nW$ with ${\sf supp}(f) \subseteq \bigcup_{m=1}^\infty  C_m$, we have $f \upharpoonright_{C_n} \neq g_n\upharpoonright_{C_n}$. By the construction of the $C_n$, we can then find $g \in {\sf Homeo}(M)$ with ${\sf supp}(g) \subseteq \bigcup_{m=1}^\infty  C_m$ and $g \upharpoonright_{C_n}= g_n \upharpoonright_{C_n}$ for all $n$. This means that $g \notin k_nW$ for all $n$, contradicting that ${\sf Homeo}(M)=\bigcup_{m=1}^\infty  k_nW$.

It thus follows that there is some set $C_n$ so that, for any $g \in {\sf Homeo}(M)$ with ${\sf supp}(g) \subseteq C_n$, there is $f \in k_nW$ with ${\sf supp}(f) \subseteq \bigcup_{m=1}^\infty  C_m$ and $f\upharpoonright_{C_n}= g\upharpoonright_{C_n}$.

Set $C=C_n$ and $D = \bigcup_{m \neq n} C_m$. Then, for all $g \in {\sf Homeo}(M)$ with ${\sf supp}(g) \subseteq C$, there is $f \in W^2$ with ${\sf supp}(f) \subseteq C \cup D$ such that $f \upharpoonright_C = g\upharpoonright_C$. Indeed, given $g$, pick $h_1, h_2 \in k_nW$ with ${\sf supp}(h_i) \subseteq C \cup D$ such that $h_1\upharpoonright_C = g\upharpoonright_C$ and $h_2 \upharpoonright_C = \text{id}\upharpoonright_C$. Then $f = h_2^{-1}h_1 \in W^2$, ${\sf supp}(f) \subseteq C \cup D$ and $f \upharpoonright_C = g \upharpoonright_C$.

Now, by using the density of $W^2$ in $V_\eps$, we can find some $h \in W^2$ such that
$$
(C\cup D)\cap h[C\cup D]=C\cap h[C]\neq \emptyset.
$$
Applying Lemma \ref{lemma:comm}, we may also choose some non-empty open $E \subseteq  C\cap h[C]$ such that for all $g \in {\sf Homeo}(M)$ with ${\sf supp}(g) \subseteq E$, there are $f_1, f_2 \in {\sf Homeo}(M)$ with ${\sf supp}(f_i) \subseteq C \cap h[C]$ and
$$
g =  f_1\inv f_2\inv f_1f_2.
$$

We now claim that 
$$
\{g \in {\sf Homeo}(M) \del {\sf supp}(g) \subseteq E \}\subseteq W^{16}.
$$ 
For this, let $g \in {\sf Homeo}(M)$ satifsfy $ {\sf supp}(g) \subseteq E$ and write $g =  f_1\inv f_2\inv f_1f_2$ as above.
Then ${\sf supp}(f_1) \subseteq C$ and ${\sf supp}(h^{-1}f_2h) \subseteq C$, so there are $q_1, q_2 \in W^2$ with ${\sf supp}(q_i) \subseteq C \cup D$ such that 
$$
q_1 \upharpoonright_C = f_1 \upharpoonright_C, \qquad  q_2 \upharpoonright_C = h^{-1}f_2h\upharpoonright_C.
$$
It follows that 
$$
{\sf supp}(hq_2h^{-1}) = h\cdot{\sf supp}(q_2) \subseteq h[C\cup  D]
$$ 
and 
$$
hq_2h^{-1}\upharpoonright_{h[C]} = f_2 \upharpoonright_{h[C]}.
$$ 
In particular, 
$$
{\sf supp}(q_1) \cap {\sf supp}(hq_2h\inv ) \subseteq (C \cup D) \cap h[C\cup D] = C \cap h[C].
$$
Thus, on their common support, $q_1$ agrees with $f_1$, whereas $hq_2h^{-1}$ agrees with $f_2$. Therefore, 
$$
g =  f_1\inv f_2\inv f_1f_2= q_1\inv   \cdot (hq_2h\inv)\inv \cdot q_1\cdot hq_2h\inv  \in W^{16},
$$
which proves the claim. Taking $B_1=E\subseteq A_1$, the lemma follows.
\end{proof}

To finish the proof that ${\sf Homeo}(M)$ is $20m$-Steinhaus,  find by Lemma \ref{lemma:coveringballs}  a covering $\{U_1, \dots, U_m\}$ of $M$ so that each $U_i$ can be written as 
$$
U_i=A_{i,1} \cup \dots \cup A_{i,p},
$$ 
where the $A_{i,j}$ are disjoint open balls of radius $<\eps/2$. By Lemma \ref{lemma:refinement}, we can find smaller non-empty open sets $B_{i,j} \subseteq A_{i,j}$ so that, for each $j=1,\ldots, m$, 
$$
\big\{ h \in {\sf Homeo}(M)\del {\sf supp}(h) \subseteq \bigcup_{i=1}^p B_{i,j}\big\}   \subseteq W^{16}.
$$
Because $W^2$ is dense in $V_\eps$, there are $f_i \in W^2$ such that $f_i[{A_{i,j}}] \subseteq B_{i,j}$. Thus, if 
${\sf supp}(g) \subseteq U_i=A_{i,1}\cup \dots \cup A_{i,p}$, then ${\sf supp}(f_igf_i^{-1})\subseteq B_{i,1}\cup \dots \cup B_{i,p}$, whereby $f_igf_i^{-1} \in W^{16}$ and  $g \in W^{20}$. It follows that
$$
\{g_1\cdots g_m \del {\sf supp}(g_i) \subseteq U_i\} \subseteq W^{20m}
$$
and so, by Theorem \ref{thm:Edwards Kirby}, $W^{20m}$ is an identity neighborhood in ${\sf Homeo}(M)$.

From the Steinhaus property, we can thus conclude the following result, which was proved in \cite{ros-sol} for $1$-manifolds, in \cite{rosendal-israel} for $2$-manifolds and in \cite{mann} for all dimensions.

\begin{thm}\cite{ ros-sol, rosendal-israel, mann}
The homeomorphism group ${\sf Homeo}(M)$ of any compact manifold $M$ has the automatic continuity property.
\end{thm}

%%%%%%%%%%%%%%%%%%%%%%%%%%%%%%%%%%%%%%%%%%%%%%%%%%%%%%%%%%%%%%%%%%%%%%%%%%%%%%%%%%%%%%%%%%%%%%%%%%%%%%%%
%%%%%%%%%%%%%%%%%%%%%%%%%%%%%%%%%%%%%%%%%%%%%%%%%%%%%%%%%%%%%%%%%%%%%%%%%%%%%%%%%%%%%%%%%%%%%%%%%%%%%%%%
%%%%%%%%%%%%%%%%%%%%%%%%%%%%%%%%%%%%%%%%%%%%%%%%%%%%%%%%%%%%%%%%%%%%%%%%%%%%%%%%%%%%%%%%%%%%%%%%%%%%%%%%
%%%%%%%%%%%%%%%%%%%%%%%%%%%%%%%%%%%%%%%%%%%%%%%%%%%%%%%%%%%%%%%%%%%%%%%%%%%%%%%%%%%%%%%%%%%%%%%%%%%%%%%%
%%%%%%%%%%%%%%%%%%%%%%%%%%%%%%%%%%%%%%%%%%%%%%%%%%%%%%%%%%%%%%%%%%%%%%%%%%%%%%%%%%%%%%%%%%%%%%%%%%%%%%%%
%%%%%%%%%%%%%%%%%%%%%%%%%%%%%%%%%%%%%%%%%%%%%%%%%%%%%%%%%%%%%%%%%%%%%%%%%%%%%%%%%%%%%%%%%%%%%%%%%%%%%%%%
%%%%%%%%%%%%%%%%%%%%%%%%%%%%%%%%%%%%%%%%%%%%%%%%%%%%%%%%%%%%%%%%%%%%%%%%%%%%%%%%%%%%%%%%%%%%%%%%%%%%%%%%
%%%%%%%%%%%%%%%%%%%%%%%%%%%%%%%%%%%%%%%%%%%%%%%%%%%%%%%%%%%%%%%%%%%%%%%%%%%%%%%%%%%%%%%%%%%%%%%%%%%%%%%%
%%%%%%%%%%%%%%%%%%%%%%%%%%%%%%%%%%%%%%%%%%%%%%%%%%%%%%%%%%%%%%%%%%%%%%%%%%%%%%%%%%%%%%%%%%%%%%%%%%%%%%%%
%%%%%%%%%%%%%%%%%%%%%%%%%%%%%%%%%%%%%%%%%%%%%%%%%%%%%%%%%%%%%%%%%%%%%%%%%%%%%%%%%%%%%%%%%%%%%%%%%%%%%%%%

\section{Automatic continuity for isometry groups}\label{section:isometry}
The goal of this section is to present some results due to I. Ben Yaacov, A. Berenstein and J. Melleray \cite{BBM}, M. Sabok \cite{sabok} and T. Tsankov \cite{tsankov} on automatic continuity properties for isometry groups of some highly homogeneous metric spaces with additional structure. For this, we are following the lead of Sabok, who devised an elegant reduction to the theory of ample generics. Some simplifications to Sabok's approach were given by M. Malicki \cite{Malicki} and our approach here discards of even more baggage by removing the need for technical concepts such as relative $\delta$-saturation, ample tuples and (weakly) isolated sequences.

\subsection{Existence of ample generics}
We begin our treatment by explaining some well-known results about ample generics in the context of automorphism groups. But let us first recall the fundamental concept originating in \cite{HHLS} and \cite{kec-ros}.
\begin{defi}
A Polish group $G$ is said to have {\em ample generics} provided that, for all $k\geqslant 1$, the action of $G$ on $G^k$ by diagonal conjugation,
$$
g\cdot (f_1,\ldots, f_k)=\big(gf_1g\inv, \ldots, gf_kg\inv\big),
$$
has a comeagre orbit.
\end{defi}
Ample generics is an interesting concept in itself, but is also a central tool in automatic continuity. Indeed, as shown in \cite{kec-ros} (see also \cite[Section 5]{rosendal-BSL} for a more topological account), every Polish group with ample generics has the automatic continuity property.

There is now quite a substantial literature on groups with ample generics. Mostly, the examples are automorphism group of countable highly homogenous model theoretical structures. In fact, for some time, it remained an open problem whether any Polish group with ample generics must be non-Archimedean, meaning that it is the automorphism group of some countable model theoretical structure. However, this was disproved by examples due to Malicki \cite{malicki2} and A. Kaïchouh and F. Le Maître \cite{kaichouh}.  Similarly, it was suspected that non-trivial locally compact groups cannot have ample generics, which turned out to be correct as established by P. Wesolek \cite{wesolek} using deep results from the theory of profinite groups.
Among the many other highlights of the theory, let us also point out the result by A. Kwiatkowska \cite{kwiatkowska} that the homeomorphism group of Cantor space has ample generics. Since this group  can alternatively be seen as the automorphism group of the countable atomless Boolean algebra, it does  falls into the category of automorphism groups.

Recall that a group action $G\curvearrowright Z$ by homeomorphisms on a topological space $Z$ is said to be {\em topologically transitive} if, for all non-empty open $W_0,W_1\subseteq Z$, there is some $g\in G$ so that $gW_0\cap W_1\neq \emptyset$. Using this, we can now provide a simple criterion for having ample generics or, more generally, for the existence of comeagre orbits under Polish group actions, cf. \cite[Lemma 9]{rosendal-jsl}.

\begin{lemme}\label{comeagre crit}
A continuous action $G\curvearrowright Z$ of a Polish group $G$ on a Polish space $Z$ has a comeagre orbit if and only if 
\begin{enumerate}
\item the action $G\curvearrowright Z$ is topologically transitive and
\item for all non-empty open $V\subseteq Z$ and identity neighbourhoods $U\subseteq G$, there is some non-empty open $W  \subseteq V$ so that the action of $U$ on $W$ is topologically transitive, i.e., for all non-empty open $W_0,W_1\subseteq W$, there is some $g\in U$ such that $gW_0\cap W_1\neq \tom$.
\end{enumerate}
\end{lemme}

The context for our discussion below is models of first-order logic, that is, structures of the form $\ku X=\langle X, \{F^\ku X_i\}_{i\in I}, \{R^\ku X_j\}_{j\in J}\rangle$, where $X$ is some set, each $F^\ku X_i\colon X^{k_i}\to X$ is a function of some finite number $k_i$ of variables and $R^\ku X_j\subseteq X^{k_j}$ is a relation of finite arity $k_j$.

\begin{defi}\label{defi:hrushovski}
A first-order structure $\ku X$ is said to have the {\em Hrushovski property} provided that, for any finite collection $\{\phi_i\}_{i\in I}$ of isomorphisms
$$
\ku A_i\maps {\phi_i }\ku B_i
$$  
between finitely generated substructures  $\ku A_i,\ku B_i\subseteq \ku X$,  there is a finitely generated substructure $\ku D\subseteq \ku X$ containing all the $\ku A_i$ and automorphisms $f_i\in {\sf Aut}(\ku X)$ so that each $f_i$ extends $\phi_i$ and leaves $\ku D$ invariant.
\end{defi}
Observe, in particular, that every structure $\ku X$ with the Hrushovski property is {\em ultrahomogeneous} in the sense that every isomorphism $\ku A\maps {\phi}\ku B$ between two finitely generated substructures of $\ku X$ extends to a full automorphism of $\ku X$ itself.

\begin{defi}
Assume also that $\ku A$, $\ku B$ and $\ku C$ are substructures of a first-order structure $\ku X$ with $\ku A\subseteq \ku B\cap \ku C$ and let $\ku D$ denote the substructure of $\ku X$ generated by $\ku B\cup \ku C$. We say that {\em $\ku B$ is  independent from $\ku C$ over $\ku A$}, written 
$$
\ku B\forkindep[{\ku A}]\ku C,
$$
provided that whenever $\phi$ and $\psi$ are automorphisms of respectively $\ku B$ and $\ku C$, both leaving $\ku A$ invariant and so that $\phi|_\ku A=\psi|_\ku A$,  then there is an automorphism $\sigma$ of $\ku D$ extending both $\phi$ and $\psi$.
\end{defi}

Observe that independence is an absolute notion in the sense that, if $\ku Y$ is a superstructure of $\ku X$,  then whether 
$$
\ku B\forkindep[{\ku A}]\ku C
$$
holds is independent of whether $\ku A$, $\ku B$ and $\ku C$ are seen as substructures of $\ku X$ or of $\ku Y$.

\begin{defi}\label{defi:extension}
A first-order structure $\ku X$ is said to have the  {\em extension property} provided that, for all finitely generated substructures $\ku A$, $\ku B$ and $\ku C$ of $\ku X$ satisfying $\ku A\subseteq \ku B\cap \ku C$, there is some $g\in {\sf Aut}(\ku X)$ so that $g|_\ku A={\sf id}_\ku A$ and 
$$
\ku B\forkindep[{\ku A}]g[\ku C].
$$
\end{defi}

Suppose $\ku X$ is a first-order structure. Then the tautological action of the  automorphism group ${\sf Aut}(\ku X)$ on the universe $X$ of $\ku X$ extends diagonally to actions on all finite powers $X^k$ or $X$. For simplicity of notation, we shall designate  finite tuples of elements of $X$ by $\ov a=(a_1,\ldots, a_k)$ etc. 

When $\ku X$ is furthermore assumed to be countable, the automorphism group ${\sf Aut}(\ku X)$ becomes a Polish group when equipped with the {\em permutation group topology} that can be described be declaring pointwise stabilisers
$$
{\sf Aut}(\ku X,\ov a)=\{g\in {\sf Aut}(\ku X)\del g\ov a=\ov a\}
$$
to be clopen subgroups. Whenever $\ku X$ is countable, this is the only topology on ${\sf Aut}(\ku X)$ that will be considered.

The next result is part of the folklore, see, e.g.,  \cite{HHLS,kec-ros,sabok}.
\begin{thm}\label{thm:ample gen}
Suppose that $\ku X$ is a countable first-order structure with the extension and Hrushovski properties. Then ${\sf Aut}(\ku X)$ has ample generics.
\end{thm}

\begin{proof}
Assume that $\ku A$ is a finitely generated substructure of $\ku X$ and that $g_1,\ldots, g_n\in {\sf Aut}(\ku X)$. We will find  a finitely generated substructure  $\ku A\subseteq \ku B\subseteq \ku X$ and automorphisms $f_1,\ldots, f_n\in {\sf Aut}(\ku X)$ satisfying the following conditions
\begin{enumerate}
\item $f_i|_\ku A=g_i|_\ku A$ for all $i$,
\item if $\ku B\subseteq \ku C\subseteq \ku X$ is a finitely generated substructure and $h_1,\ldots, h_n, h_1', \ldots, h_n'\in {\sf Aut}(\ku X)$ are automorphisms so that 
$$
h_i|_\ku B=h'_i|_\ku B=f_i|_\ku B
$$ 
for all $i$, then there are $p, q_i\in {\sf Aut}(\ku X)$ so that $p|_\ku A={\sf id}_\ku A$,  whereas
$$
q_i|_\ku C=h_i|_\ku C,     \qquad       p\inv q_ip |_\ku C=h'_i|_\ku C
$$
for all $i$.
\end{enumerate}
Furthermore, if $\ku A$ is the structure generated by the empty set, we choose $\ku B=\ku A$.
Doing this suffices to verify the  criterion of Lemma \ref{comeagre crit} for ${\sf Aut}(\ku X)$ having ample generics. 

To find $\ku B$ and the $f_i$, we apply the Hrushovski property to the collection of isomorphisms $\{g_i|_\ku A\}_{i=1}^n$ to give us a finitely generated substructure $\ku A\subseteq \ku B\subseteq \ku X$ and automorphisms $f_1,\ldots, f_n\in {\sf Aut}(\ku X)$ so that 
$$
f_i[\ku B]=\ku B
$$
and 
$$
f_i|_\ku A=g_i|_\ku A
$$
for all $i$. Observe also that, if $\ku A$ is the substructure generated by $\emptyset$, we may simply set $\ku B=\ku A$.

Assume now that $\ku B\subseteq \ku C\subseteq \ku X$ is a finitely generated substructure and  $h_i, h_i'\in {\sf Aut}(\ku X)$ 
are automorphisms so that $h_i|_\ku B=h'_i|_\ku B=f_i|_\ku B$. We apply the Hrushovski property again to the family $\{h_i|_\ku C, h_i'|_\ku C\}_i$ to find a finitely generated substructure $\ku C\subseteq \ku D\subseteq \ku X$ and automorphisms $k_i, k_i'\in {\sf Aut}(\ku X)$ so that
$$
k_i[\ku D]=k_i'[\ku D]=\ku D
$$
and 
$$
k_i|_\ku C=h_i|_\ku C, \qquad k'_i|_\ku C=h'_i|_\ku C
$$
for all $i$. By the extension property, we then find some $p\in {\sf Aut}(\ku X)$ so that $p|_\ku B={\sf id}_\ku B$ and 
$$
\ku D\forkindep[{\ku B}]p[\ku D].
$$
Because $\ku B\subseteq \ku C\subseteq \ku D$, $p|_\ku B={\sf id}_\ku B$ and $f_i[\ku B]=\ku B$, we have that
$$
pk_i'p\inv |_\ku B            
=pk_i' |_\ku B
=ph_i' |_\ku B
=pf_i |_\ku B
=f_i|_\ku B
=h_i|_\ku B
=k_i|_\ku B.
$$
Thus, as $k_i$ and $pk_i'p\inv$ leave respectively $\ku D$ and $p[\ku D]$ invariant and as $\ku D\forkindep[{\ku B}]p[\ku D]$, we can by ultrahomogenity of $\ku X$ find automorphisms $q_i\in {\sf Aut}(\ku X)$ extending both $k_i|_\ku D$ and  $pk_i'p\inv |_{p[\ku D]}$. This means that $p\inv q_ip$ extends $k'_i|_\ku D$ and so
$$
q_i|_\ku C=k_i|_\ku C=h_i|_\ku C,     \qquad       p\inv q_ip |_\ku C=k'_i|_\ku C=h'_i|_\ku C
$$
as required.
\end{proof}

%%%%%%%%%%%%%%%%%%%%%%%%%%%%%%%%%%%%%%%%%%%%%%%%%%%%%%%%%%%%%%%%%%%%%%%%%%%%%%%%%%%%%%%%%%%%%%%%%%%%%%%%
%%%%%%%%%%%%%%%%%%%%%%%%%%%%%%%%%%%%%%%%%%%%%%%%%%%%%%%%%%%%%%%%%%%%%%%%%%%%%%%%%%%%%%%%%%%%%%%%%%%%%%%%
%%%%%%%%%%%%%%%%%%%%%%%%%%%%%%%%%%%%%%%%%%%%%%%%%%%%%%%%%%%%%%%%%%%%%%%%%%%%%%%%%%%%%%%%%%%%%%%%%%%%%%%%
%%%%%%%%%%%%%%%%%%%%%%%%%%%%%%%%%%%%%%%%%%%%%%%%%%%%%%%%%%%%%%%%%%%%%%%%%%%%%%%%%%%%%%%%%%%%%%%%%%%%%%%%
%%%%%%%%%%%%%%%%%%%%%%%%%%%%%%%%%%%%%%%%%%%%%%%%%%%%%%%%%%%%%%%%%%%%%%%%%%%%%%%%%%%%%%%%%%%%%%%%%%%%%%%%
%%%%%%%%%%%%%%%%%%%%%%%%%%%%%%%%%%%%%%%%%%%%%%%%%%%%%%%%%%%%%%%%%%%%%%%%%%%%%%%%%%%%%%%%%%%%%%%%%%%%%%%%
%%%%%%%%%%%%%%%%%%%%%%%%%%%%%%%%%%%%%%%%%%%%%%%%%%%%%%%%%%%%%%%%%%%%%%%%%%%%%%%%%%%%%%%%%%%%%%%%%%%%%%%%
%%%%%%%%%%%%%%%%%%%%%%%%%%%%%%%%%%%%%%%%%%%%%%%%%%%%%%%%%%%%%%%%%%%%%%%%%%%%%%%%%%%%%%%%%%%%%%%%%%%%%%%%
%%%%%%%%%%%%%%%%%%%%%%%%%%%%%%%%%%%%%%%%%%%%%%%%%%%%%%%%%%%%%%%%%%%%%%%%%%%%%%%%%%%%%%%%%%%%%%%%%%%%%%%%
%%%%%%%%%%%%%%%%%%%%%%%%%%%%%%%%%%%%%%%%%%%%%%%%%%%%%%%%%%%%%%%%%%%%%%%%%%%%%%%%%%%%%%%%%%%%%%%%%%%%%%%%

\subsection{First-order metric structures}
With the prerequisites on ample generics out of the way, we now turn to first-order metric structures. For this, consider the countable first-order relational language 
$$
L_{\sf dist}=\big\{D_r\del r\in \Q_+\big\},
$$
where each $D_r$ is a binary relational symbol and $\Q_+=\Q\cap [0,\infty[$. An $L_{\sf dist}$-structure $\ku X=\langle X, \{D^\ku X_r\}_{r\in \Q_+}\rangle 
$ is said to be a {\em metric $L_{\sf dist}$-structure} provided that there is a metric $d$ on $X$ so that
$$
\ku X\models D_r(x,y) \quad\equi\quad d(x,y)\leqslant r.
$$
Observe that, in this case, the metric $d$ and the family of interpretations $ \{D^\ku X_r\}_{r\in \Q_+}$ are interdefinable. Note also that, any substructure $\ku Y$ of a metric $L_{\sf dist}$-structure $\ku X$ is again a metric $L_{\sf dist}$-structure and that the associated metric on $Y$ is simply the restriction of that on $X$. 

More generally, if $L$ is a first-order language containing $L_{\sf dist}$, then an $L$-structure $\ku X$ will be said to be a {\em metric $L$-structure} provided that the following conditions hold.
\begin{enumerate}
\item The $L_{\sf dist}$-reduct $\ku X|_{L_{\sf dist}}$ is a metric $L_{\sf dist}$-structure,
\item for every $k$-ary function symbol $F\in L$, the interpretation 
$$
F^\ku X\colon X^k\to X
$$
is continuous with respect to the metric topology on $X$, 
\item for every $k$-ary relation symbol $R\in L$, the interpretation 
$$
R^\ku X=\{(x_1,\ldots, x_k)\in X^k\del \ku X\models R(x_1,\ldots, x_k)\}
$$
is closed with respect to the metric topology on $X$.
\end{enumerate}
For good measure, although this will play no role in what follows, let us remark that the class of metric $L$-structures is not first-order axiomatisable. Nevertheless, it is axiomatisable in the logic $L_{\om_1\om}$.
Again, every substructure of a metric $L$-structure is also a metric $L$-structure. 

A metric $L$-structure $\ku X$ will be said to be {\em separable} or {\em complete} provided that the universe $X$ of $\ku X$ is respectively separable or complete with respect to the associated metric $d$. Let trivially, we say that a substructure $\ku Y \subseteq \ku X$ is {\em superdense} in $\ku X$ provided that the universe $Y$ of $\ku Y$ is $d$-dense in $X$ and, furthermore, that $R^\ku Y$ is $d$-dense in $R^\ku X$ for  every relation symbol $R\in L$. 

\begin{lemme}\label{lem:ext}
Suppose $\ku Y$ is a superdense substructure of  a complete metric $L$-structure  $\ku X$. Then every automorphism $g\in {\sf Aut}(\ku Y)$ extends uniquely to an automorphism $\hat g \in {\sf Aut}(\ku X)$.
\end{lemme}

\begin{proof} 
Suppose that an automorphism $g\in {\sf Aut}(\ku Y)$ is given. Let $X$ and $Y$ denote the universes of $X$ and $Y$ respectively and $d$ be the metric associated with $X$. Then, since $g$ is an automorphism of $\ku Y$ and $d|_Y$ is definable from $\ku Y$, 
$g$ is an autoisometry of $(Y,d)$. Because $(X,d)$ is complete and $Y$ is dense in $X$, this implies that $g$ extends uniquely to an isometry $\hat g$ of $(X,d)$. We claim that $\hat g$ is is an automorphism of $\ku X$. Indeed, if $F\in L$ is a $k$-ary function symbol, its interpretation $F^\ku X\colon X^k\to X$ is continuous and so, if $\ov a$ is a $k$-tuple in $\ku X$, we can find tuples $\ov a_n$ in $\ku Y$ converging to $\ov a$, whereby
\maths{
F^\ku X(\hat g\ov a)
&=F^\ku X (\hat g\cdot \lim_n\ov a_n)\\
&=F^\ku X ( \lim_ng\ov a_n)\\
&=\lim_n F^\ku X(g\ov a_n)\\
&=\lim_ng\big(F^\ku X(\ov a_n)\big)\\
&=\hat g\big(\lim_nF^\ku X(\ov a_n)\big)\\
&=\hat g\big(F^\ku X(\lim_n\ov a_n)\big)\\
&=\hat g\big(F^\ku X(\ov a)\big).
}
Similarly, if $R\in L$ is a $k$-ary relation symbol and $\ov a\in R^\ku X$,  we can pick $\ov a_n\in  R^\ku Y$ converging to $\ov a$, whereby $g\ov a_n\in R^\ku Y$ and hence 
$$
\hat g\ov a=\lim_ng\ov a_n\in R^\ku X
$$
as $R^\ku X$ is closed in $X^k$. Using that also $\hat g\inv$ preserves $R^\ku X$, we conclude that $\hat g$ is an automorphism of $\ku X$. 
\end{proof}
Because the extension operator $g\mapsto \hat g$ is clearly an injective group homomorphism, this shows that, under the assumptions of Lemma \ref{lem:ext}, we may simply identify ${\sf Aut}(\ku Y)$ with the subgroup
$$
\big\{g\in {\sf Aut}(\ku X)\del g[\ku Y]=\ku Y\big\}
$$
of ${\sf Aut}(\ku X)$.

The next lemma is obtained via a straightforward Löwenheim--Skolem construction.
\begin{lemme}\cite{sabok}\label{lem:skolem}
Let $\ku X$ be a separable complete metric $L$-structure in a countable language $L\supseteq L_{\sf dist}$ and assume that $\ku X$ has the extension and Hrushovski properties. Suppose also that, for every finite tuple $\ov a$ in $\ku X$,  we are given a countable subset $F_{\ov a}\subseteq {\sf Aut}(\ku X)$. 
Then there is a countable superdense substructure $\ku Y\subseteq \ku X$ with the extension and Hrushovski properties and that is furthermore invariant under all elements of $F_{\ov a}$ whenever $\ov a$ is a tuple in $\ku Y$.
\end{lemme}

\begin{proof}
Because the language $L$ is countable and $\ku X$ is separable, it contains a countable superdense substructure $\ku X_1\subseteq \ku X$. By induction on $n\geqslant 1$, we define an increasing sequence of countable substructures
$$
\ku X_1\subseteq \ku X_2\subseteq {\ku X}_3\subseteq \ldots\subseteq {\ku X}
$$
and countable subsets
$$
F_\tom=A_1\subseteq A_2\subseteq \ldots \subseteq {\sf Aut}(\ku X)
$$
so that the following properties are satisfied.
\begin{enumerate}
\item $F_{\ov a}\subseteq A_{n+1}$ for all  tuples $\ov a$ in $\ku X_n$,
\item if $\ku A, \ku B, \ku C$ are finitely generated substructures of $\ku X_n$ with  $\ku A\subseteq \ku B\cap \ku C$, then there is some $f\in A_{n+1}$ so that $f|_\ku A={\sf id}|_\ku A$ and $\ku B\forkindep[{\ku A}]f[\ku C]$,
\item if $\big\{\ku A_i\maps {\phi_i }\ku B_i\big\}_{i\in I}$ is a finite collection of isomorphisms between finitely generated substructures  $\ku A_i,\ku B_i\subseteq \ku X_n$, then there is a finitely generated substructure $\ku D\subseteq \ku X_{n+1}$, containing all $\ku A_i$, and automorphisms $f_i\in A_{n+1}$ so that each $f_i$ extends $\phi_i$ and leaves $\ku D$ invariant,
\item $f[\ku X_{n+1}]=\ku X_{n+1}$ for all  $f\in A_{n}$. 
\end{enumerate}
For the inductive construction, one supposes that $\ku X_1\,\ldots, \ku X_m$ and $A_1,\ldots, A_m$ have been defined so that (1)-(4) hold for all $n<m$. Then the extension and Hrushovski properties of $\ku X$, along with the countability of $L$,  $\ku X_m$ and $A_m$, ensure that one may construct a pair $(A_{m+1},\ku X_{m+1})$ so that (1)-(4) now hold for all $n\leqslant m$.

Assuming that the construction has been done, we note that the union $\ku Y=\bigcup_{n=1}^\infty \ku X_n$ will be a countable superdense substructure of $\ku X$ closed under all elements of $\bigcup_{n=1}^\infty A_n$. It then follows from (2) and (3) that $\ku Y$ has the extension and Hrushovski properties respectively and from (1) that $\ku Y$ is invariant under all elements of $F_{\ov a}$ whenever $\ov a$ is a tuple in $\ku Y$.
\end{proof}

We are now ready to prove a slight strengthening of the central lemma of Sabok's paper \cite[Lemma 6.1]{sabok} that, in a single step, will allow us to deduce almost all consequences of ample generics. 

\begin{lemme}\label{lem:sabok}
Let $\ku X$ be a separable complete metric $L$-structure in a countable language $L\supseteq L_{\sf dist}$ and assume that $\ku X$ has the extension and Hrushovski properties. Suppose also that elements $k_n,h_n\in {\sf Aut}(\ku X)$ and subsets $W_n\subseteq {\sf Aut}(\ku X)$ are given so that $\bigcup_{n=1}^\infty k_nW_nh_n$ is  comeagre in ${\sf Aut}(\ku X)$. Then there is a finite tuple $\ov a$ in $\ku X$ and some $n$ so that
$$
{\sf Aut}(\ku X, \ov a)\;\subseteq\; \big(W_n\inv W_n\big)^{10}.
$$
\end{lemme}

\begin{proof}
Suppose for a contradiction that the conclusion of the lemma fails.  For every tuple $\ov a$ and number $n$, we may then pick some 
$$
f_{\ov a,n}\in {\sf Aut}(\ku X, \ov a)\setminus \big(W_n\inv W_n\big)^{10}.
$$
By Lemma \ref{lem:skolem}, there is a countable  superdense substructure $\ku Y\subseteq \ku X$ with the extension and Hrushovski properties and that is invariant under all $h_n$ and under every $f_{\ov a,n}$ whenever $\ov a$ is a tuple in $\ku Y$.
Thus, by Theorem \ref{thm:ample gen} and Lemma \ref{lem:ext}, the  subgroup
$$
G=\big\{g\in {\sf Aut}(\ku X)\del g[\ku Y]=\ku Y\big\}
$$
has ample generics when equipped with the finer Polish permutation group topology induced by its action on $\ku Y$. Observe that $h_n, f_{\ov a,n}\in G$ for all $n$ and tuples $\ov a$ in $\ku Y$.

Observe that the group multiplication map restricts to a continuous open surjection 
$$
{\sf Aut}(\ku X)\times G\maps{\sf m} {\sf Aut}(\ku X).
$$ 
It thus follows that the inverse image ${\sf m}\inv\Big(\bigcup_{n=1}^\infty k_nW_nh_n\Big)$ is comeagre in the Polish ${\sf Aut}(\ku X)\times G$. Therefore, by the Kuratowski--Ulam theorem, there is a comeagre set of $t\in {\sf Aut}(\ku X)$, so that the set of $g\in G$ for which $tg\in \bigcup_{n=1}^\infty k_nW_nh_n$ is comeagre. Fix some such $t\in {\sf Aut}(\ku X)$. Then
$$
\bigcup_{n=1}^\infty \Big(G\cap t\inv k_nW_nh_n \Big) \text{ is comeagre in  } G.
$$
Note that, if we pick $g_n\in G\cap \big(t\inv k_nW_nh_n\big)$ whenever the latter set is non-empty and otherwise set $g_n=1$, then, because $g_n,h_n\in G$, we have that
\maths{
G\cap \big(t\inv k_nW_nh_n\big)
\;&\subseteq\;  G\cap   \big(g_n\cdot h_n\inv W_n\inv k_n\inv t\cdot     t\inv k_nW_nh_n \big)\\
&=\;G\cap \big(g_nh_n\inv W_n\inv W_nh_n \big)\\
&=\; g_nh_n\inv \big(G\cap W_n\inv W_n \big)h_n
}
for all $n$. In therefore follows that
$$
\bigcup_{n=1}^\infty g_nh_n\inv \big(G\cap W_n\inv W_n \big)h_n
$$
is comeagre in $G$. 

Because $G$ has ample generics, it follows from \cite[Theorem 5.10]{rosendal-BSL}, along with the first line of its proof, that 
\maths{
\big\{g\in {\sf Aut}(\ku X)\del g[\ku Y]=\ku Y\;\&\; g\ov a=\ov a\big\}
&\subseteq \big(G\cap W_n\inv W_n \big)^{10}\subseteq \big(W_n\inv W_n\big)^{10}
}
for some tuple $\ov a$ in $\ku Y$ and some number $n$. However, this contradicts that 
$$
f_{\ov a,n}\notin  \big(W_n\inv W_n\big)^{10}.
$$
whereas $f_{\ov a,n}\in \big\{g\in {\sf Aut}(\ku X)\del g[\ku Y]=\ku Y\;\&\; g\ov a=\ov a\big\}$.
\end{proof}

Thus far, our discussion has not involved any topology on ${\sf Aut}(\ku X)$,  but of course for results on automatic continuity we need to see it as a topological group. So suppose  $\ku X$ is  a separable complete metric $L$-structure in a countable language $L\supseteq L_{\sf dist}$, let $X$ denote the universe of $\ku X$ and $d$ the associated complete metric on $X$. Since the metric space $(X,d)$ is separable and complete, the group ${\sf Isom}(X,d)$ of autoisometries of $X$ is a Polish group when equipped with the topology of pointwise convergence \cite[Section 9B]{kechris-book}. But furthermore, because the functions $F^\ku X$ of the structure $\ku X$ are continuous and the relations $R^\ku X$ closed, we see that ${\sf Aut}(\ku X)$ is a closed subgroup of ${\sf Isom}(X,d)$ and thus is a Polish group in its own right. When dealing with uncountable separable complete metric $L$-structures, ${\sf Aut}(\ku X)$ will henceforth be given this Polish topology of pointwise convergence on $(X,d)$. 

Note that the metric $d$ on $X$ induces a compatible metric $d_\infty$ on every finite power $X^k$ by the formula
$$
d_\infty(\ov a,\ov b)=\max_id(a_i,b_i)
$$
and observe that a neighbourhood basis at the identity of ${\sf Aut}(\ku X)$ is given by sets of the form
$$
N(\ov a, \eps)=\{g\in {\sf Aut}(\ku X)\del d_\infty(g\ov a,\ov a)<\eps \}
$$
where $\ov a$ is a finite tuple in $\ku X$ and $\eps>0$. Let also 
$$
\ku O(\ov a)=\big\{ g\ov a\del g\in {\sf Aut}(\ku X)\big\}
$$
denote the orbit of the tuple $\ov a$ under the action of ${\sf Aut}(\ku X)$ and write
$$
\ku O(\ov a,\eps)=\{\ov b\in \ku O(\ov a)\del d_\infty(\ov a, \ov b)<\eps\}=N(\ov a, \eps)\!\cdot\! \ov a.
$$
Moreover, 
$$
\ku O(\ov b/\ov a)=\big\{\ov c\del \ku O(\ov c,\ov a)=\ku O(\ov b,\ov a)\big\}={\sf Aut}(\ku X,\ov a)\!\cdot\!\ov b.
$$
Finally, a family $\ku B\subseteq \bigcup_{k\geqslant 1}X^k$ is said to be a {\em basis} for ${\sf Aut}(\ku X)$ provided that, for all tuples $\ov a$ and $\eps>0$, there are $\ov b\in \ku B$  and $\eta>0$ so that
$$
{\sf Aut}(\ku X,\ov b)\subseteq {\sf Aut}(\ku X,\ov a) \qquad \text{and}\qquad N(\ov b,\eta)\subseteq N(\ov a,\eps).
$$

\begin{thm}\label{thm:autconti}
Let $\ku X$ be a separable complete metric $L$-structure in a countable language $L\supseteq L_{\sf dist}$ and assume that $\ku X$ has the extension and Hrushovski properties. Fix also a basis $\ku B$ for  ${\sf Aut}(\ku X)$ and suppose also that, for all $\ov a\in \ku B$ and $\eps>0$, 
$$
{\sf int}\big\{ (g,f)\in N(\ov a,\eps)\times {\sf Aut}(\ku X)\del O(\ov a/g\ov a)\cap \ku O(f\ov a/\ov a)\neq \tom\big\}\;\neq\;\tom.
$$
Then whenever ${\sf Aut}(\ku X)=\bigcup_{n=1}^\infty k_nW_nh_n$ is a covering by two-sided translates of subsets $W_n\subseteq {\sf Aut}(\ku X)$, we have
$$
{\sf int}\Big(\big(W_n\inv W_n\big)^{32}\Big)\neq \tom
$$
for some $n$.
\end{thm}

\begin{proof}
Fix a covering ${\sf Aut}(\ku X)=\bigcup_{n=1}^\infty k_nW_nh_n$. By leaving out all terms so that $W_n$ is meagre, we can instead suppose that $\bigcup_{n=1}^\infty k_nW_nh_n$ is comeagre in ${\sf Aut}(\ku X)$ and that every $W_n$ is non-meagre. By Lemma \ref{lem:sabok} there is some tuple $\ov a$ and some number $n$ so that 
$$
{\sf Aut}(\ku X, \ov a)\subseteq \big(W_n\inv W_n\big)^{10}.
$$
Because $W_n$ is  non-meagre, it is somewhere dense and thus $W_n\inv W_n$ is dense in some identity neighbourhood. Thus, by enlarging $\ov a$ if necessary, we can suppose that for some $\eps>0$, $W_n\inv W_n$ is dense in $N(\ov a,\eps)$. Furthermore, without loss of generality, we can suppose that $\ov a\in \ku B$. 

By the assumptions in the theorem,  there is a non-empty open set $U \subseteq  N(\ov a,\eps)\times {\sf Aut}(\ku X)$ so that $O(\ov a/g\ov a)\cap \ku O(f\ov a/\ov a)\neq \tom$ for all $(g,f)\in U$. So, by the density of  $W_n\inv W_n$  in $N(\ov a,\eps)$, we can find $g\in N(\ov a,\eps)\cap W_n\inv W_n$, so that the section
$$
U_g=\{f \in {\sf Aut}(\ku X)\del (g,f)\in U\}
$$
is non-empty open. Observe however that, for  $f\in U_g$, there are 
$$
h\in {\sf Aut}(\ku X, g\ov a )=g\cdot {\sf Aut}(\ku X,\ov a)\cdot g\inv \subseteq \big(W_n\inv W_n\big)^{12}
$$
and 
$$
k\in {\sf Aut}(\ku X,\ov a)\subseteq \big(W_n\inv W_n\big)^{10}
$$
so that $h\ov a=kf\ov a$ and hence 
$$
f\in k\inv h\cdot {\sf Aut}(\ku X,\ov a)\subseteq \big(W_n\inv W_n\big)^{32}.
$$
In other words, $\tom\neq U_g\subseteq\big(W_n\inv W_n\big)^{32}$.
\end{proof}

\begin{cor}\label{cor:autconti}
Let $\ku X$ be a separable complete metric $L$-structure in a countable language $L\supseteq L_{\sf dist}$ and assume that $\ku X$ has the extension and Hrushovski properties. Fix also a basis $\ku B$ for  ${\sf Aut}(\ku X)$ and suppose also that, for all $\ov a\in \ku B$ and $\eps>0$, 
$$
{\sf int}\big\{ (g,f)\in N(\ov a,\eps)\times {\sf Aut}(\ku X)\del\ku O(\ov a/g\ov a)\cap \ku O(f\ov a/\ov a)\neq \tom\big\}\;\neq\;\tom.
$$
Then  ${\sf Aut}(\ku X)$ has the automatic continuity property.
\end{cor}

Let us remark that the proof of Corollary \ref{cor:autconti} shows that ${\sf Aut}(\ku X)$  is actually a Steinhaus group with coefficient $k=64$.

%%%%%%%%%%%%%%%%%%%%%%%%%%%%%%%%%%%%%%%%%%%%%%%%%%%%%%%%%%%%%%%%%%%%%%%%%%%%%%%%%%%%%%%%%%%%%%%%%%%%%%%%
%%%%%%%%%%%%%%%%%%%%%%%%%%%%%%%%%%%%%%%%%%%%%%%%%%%%%%%%%%%%%%%%%%%%%%%%%%%%%%%%%%%%%%%%%%%%%%%%%%%%%%%%
%%%%%%%%%%%%%%%%%%%%%%%%%%%%%%%%%%%%%%%%%%%%%%%%%%%%%%%%%%%%%%%%%%%%%%%%%%%%%%%%%%%%%%%%%%%%%%%%%%%%%%%%
%%%%%%%%%%%%%%%%%%%%%%%%%%%%%%%%%%%%%%%%%%%%%%%%%%%%%%%%%%%%%%%%%%%%%%%%%%%%%%%%%%%%%%%%%%%%%%%%%%%%%%%%
%%%%%%%%%%%%%%%%%%%%%%%%%%%%%%%%%%%%%%%%%%%%%%%%%%%%%%%%%%%%%%%%%%%%%%%%%%%%%%%%%%%%%%%%%%%%%%%%%%%%%%%%
%%%%%%%%%%%%%%%%%%%%%%%%%%%%%%%%%%%%%%%%%%%%%%%%%%%%%%%%%%%%%%%%%%%%%%%%%%%%%%%%%%%%%%%%%%%%%%%%%%%%%%%%
%%%%%%%%%%%%%%%%%%%%%%%%%%%%%%%%%%%%%%%%%%%%%%%%%%%%%%%%%%%%%%%%%%%%%%%%%%%%%%%%%%%%%%%%%%%%%%%%%%%%%%%%
%%%%%%%%%%%%%%%%%%%%%%%%%%%%%%%%%%%%%%%%%%%%%%%%%%%%%%%%%%%%%%%%%%%%%%%%%%%%%%%%%%%%%%%%%%%%%%%%%%%%%%%%
%%%%%%%%%%%%%%%%%%%%%%%%%%%%%%%%%%%%%%%%%%%%%%%%%%%%%%%%%%%%%%%%%%%%%%%%%%%%%%%%%%%%%%%%%%%%%%%%%%%%%%%%
%%%%%%%%%%%%%%%%%%%%%%%%%%%%%%%%%%%%%%%%%%%%%%%%%%%%%%%%%%%%%%%%%%%%%%%%%%%%%%%%%%%%%%%%%%%%%%%%%%%%%%%%

\section{The isometry group of the Urysohn space}
The {\em Urysohn metric space} $\U$ first constructed by P. S. Urysohn in \cite{Urysohn1,Urysohn2} is a separable complete metric space satisfying the following {\em metric extension property}.
\begin{quote}
For any finite metric space  $X$, any subspace $Y\subseteq X$ and any isometric embedding 
$$
Y\overset\phi\longrightarrow \U,
$$ 
there exists an extension $\tilde \phi$ of $\phi$ to an isometric embedding 
$$
X\overset{\tilde\phi}\longrightarrow \U.
$$
\end{quote}
In fact, as shown by Urysohn, separability, completeness and the extension property completely determine $\U$ up to isometry and also imply that $\U$ is universal for all separable metric spaces, i.e., $\U$ contains an isometric copy of every separable metric space.

Consider now $\U$ as a complete metric structure in the language $L_{\sf dist}$ and note that its automorphism group is nothing but the group ${\sf Isom}(\U)$ of all isometries of $\U$ equipped with the topology of pointwise convergence on $\U$. Since finitely generated substructures of $\U$ are just finite metric spaces, the main result of \cite{solecki-isom} simply states that $\U$ has the Hrushovski property. Observe also that, if $X$, $Y$ and $Z$ are finite subspaces of $\U$ so that $X\subseteq Y\cap Z$ and so that
$$
d(y,z)=\min_{x\in X}d(y,x)+d(x,z)
$$
for all $y\in Y$ and $z\in Z$,  then $Y\forkindep_XZ$. A simple application of the metric extension property above thus implies that $\U$ has the extension property as a first-order structure.

We will now proceed to show that the automorphism group, i.e., ${\sf Isom}(\U)$ satisfies the criteria of Corollary \ref{cor:autconti} and thus has the automatic continuity property. For this, we first establish the following lemma. We note that \cite[Lemma 5.1]{Malicki} plays a similar role in \cite{Malicki}, but the proof provided there is not quite correct and needs an argument like ours below. The approach of \cite{sabok,Sabok-err} seems more involved.

\begin{lemme}\label{lem:metric}
Suppose $X=\{a_i,c_j,e_k\}_{i,j,k=1}^n$ is a set of cardinality $3n$ and that 
$$
d\colon X\times X\to [0,\infty[ 
$$
is a function so that the restriction of $d$ to each of the sets $\{a_i,c_j\}_{i,j}$ and  $\{a_i,e_j\}_{i,j}$ is a metric. Assume furthermore that there is some $\delta>0$ so that, for all $i,j,k$,
\begin{enumerate}
\item\label{2} $d(a_i,c_j)+d(c_j,a_k)>d(a_i,a_k)+\delta$,
\vspace{.1cm}
\item\label{3} $\delta>d(a_k,e_k)$,
\vspace{.1cm}
\item\label{4} $d(a_i,e_j)> d(a_i,a_j)$,
\vspace{.1cm}
\item\label{5} $d(a_i,a_j)=d(c_i,c_j)=d(e_i,e_j)$,
\vspace{.1cm}
\item\label{6} $d(c_j,a_k)=d(c_j,e_k)$.
\end{enumerate}
Then $d$ is a metric on $X$.
\end{lemme}

\begin{proof}
Since the restriction of $d$ to $\{a_i,c_j\}_{i,j}$ is a metric, it follows from \eqref{5} and \eqref{6} that also the restriction  to $\{e_i,c_j\}_{i,j}$ is a metric. It follows from this that it only remains to verify the triangle inequality for triples consisting of one point from each of the sets $\{a_i\}_{i}$, $\{c_i\}_{i}$ and $\{e_i\}_{i}$. This task splits into three cases according to whether the middle point belongs to the first, second or third set.

Let $i,i,k$ be arbitrary. Then
\maths{
d(a_i,c_j)+d(c_j,e_k)
&\stackrel{\eqref{6}}{=} d(a_i,c_j)+d(c_j,a_k)\\
&\stackrel{\eqref{2}}{>} d(a_i,a_k)+\delta\\
&\stackrel{\eqref{3}}{>} d(a_i,e_k),\\
}
\maths{
d(a_i,e_j)+d(e_j,c_k)
&\stackrel{\eqref{6}}{=} d(a_i,e_j)+d(a_j,c_k)\\
&\stackrel{\eqref{4}}{>} d(a_i,a_j)+d(a_j,c_k)\\
&\geqslant d(a_i,c_k),\\
}
and
\maths{
d(e_i,a_j)+d(a_j,c_k)
&\stackrel{\eqref{6}}{=} d(e_i,a_j)+d(e_j,c_k)\\
&\stackrel{\eqref{4}}{>} d(a_i,a_j)+d(e_j,c_k)\\
&\stackrel{\eqref{5}}{=}  d(e_i,e_j)+d(e_j,c_k)\\
&\geqslant d(e_i,c_k)\\
}
which confirms these three triangle inequalities.
\end{proof}

\begin{thm}\cite{sabok} The group ${\sf Isom}(\U)$ has the automatic continuity property.
\end{thm}

\begin{proof}
Observe that the collection $\ku B$ of all finite tuples $\ov a=(a_1,\ldots, a_n)$ of distinct points $a_i\in \U$ is a basis for ${\sf Isom}(\U)$. So suppose $\ov a\in \ku B$ and $\eps>0$ are given and let $U$ denote the set of all $(g,f)\in N(\ov a,\eps)\times{\sf Isom}(\U)$ so that, for some $\delta>0$ and all $i,j,k$, we have
\begin{enumerate}
\item[(i)] $d(a_i,ga_j)+d(ga_j,a_k)>d(a_i,a_k)+\delta$,
\vspace{.1cm}
\item[(ii)] $\delta>d(a_k,fa_k)$,
\vspace{.1cm}
\item[(iii)] $d(a_i,fa_j)> d(a_i,a_j)$.
\end{enumerate}
Observe that, by the metric extension property, there is some $g\in N(\ov a,\eps)$ so that $$
d(a_i,ga_j)+d(ga_j,a_k)>d(a_i,a_k)
$$ for all $i,j,k$. Choosing $\delta>0$ small enough, 
we may ensure that also (ii) holds. Picking then $f\in N(\ov a,\delta)$ so that (iii) is satisfied, we finally have $(g,f)\in U$, showing that $U$ is non-empty and evidently also open.  Furthermore, suppose that $(g,f)\in U$ and write $c_i=ga_i$. Then by Lemma \ref{lem:metric} the following conditions
\begin{enumerate}
\item\label{a} $d(a_i,c_j)+d(c_j,a_k)>d(a_i,a_k)+\delta$,
\vspace{.1cm}
\item\label{b} $\delta>d(a_k,fa_k)=d(a_k,e_k)$,
\vspace{.1cm}
\item\label{c} $d(a_i,e_j)=d(a_i,fa_j)> d(a_i,a_j)$,
\vspace{.1cm}
\item\label{d} $d(a_i,a_j)=d(c_i,c_j)=d(e_i,e_j)$,
\vspace{.1cm}
\item\label{e} $d(c_j,a_k)=d(c_j,e_k)$
\end{enumerate}
determine a metric on the set $\{a_i,c_j,e_k\}_{i,j,k=1}^n$ that agrees with the metric from $\U$ on the subset $\{a_i,c_j\}_{i,j=1}^n$. It thus follows from the metric extension property, that we can identify the points $e_i$ with actual points in $\U$. Thus, by ultrahomogeneity of $\U$ and condition \eqref{e}, we have
$$
\ov e\in \ku O(\ov a/\ov c)=\ku O(\ov a/g\ov a),
$$
whereas, from ultrahomogeneity and condition \eqref{c}, we have 
$$
\ov e\in \ku O(f\ov a/\ov a).
$$
In particular, $\ku O(\ov a/g\ov a)\cap \ku O(f\ov a/\ov a)\neq \tom$. This shows that ${\sf Isom}(\U)$ satisfies the criteria of Corollary \ref{cor:autconti} and therefore has the automatic continuity property
\end{proof}

%%%%%%%%%%%%%%%%%%%%%%%%%%%%%%%%%%%%%%%%%%%%%%%%%%%%%%%%%%%%%%%%%%%%%%%%%%%%%%%%%%%%%%%%%%%%%%%%%%%%%%%%
%%%%%%%%%%%%%%%%%%%%%%%%%%%%%%%%%%%%%%%%%%%%%%%%%%%%%%%%%%%%%%%%%%%%%%%%%%%%%%%%%%%%%%%%%%%%%%%%%%%%%%%%
%%%%%%%%%%%%%%%%%%%%%%%%%%%%%%%%%%%%%%%%%%%%%%%%%%%%%%%%%%%%%%%%%%%%%%%%%%%%%%%%%%%%%%%%%%%%%%%%%%%%%%%%
%%%%%%%%%%%%%%%%%%%%%%%%%%%%%%%%%%%%%%%%%%%%%%%%%%%%%%%%%%%%%%%%%%%%%%%%%%%%%%%%%%%%%%%%%%%%%%%%%%%%%%%%
%%%%%%%%%%%%%%%%%%%%%%%%%%%%%%%%%%%%%%%%%%%%%%%%%%%%%%%%%%%%%%%%%%%%%%%%%%%%%%%%%%%%%%%%%%%%%%%%%%%%%%%%
%%%%%%%%%%%%%%%%%%%%%%%%%%%%%%%%%%%%%%%%%%%%%%%%%%%%%%%%%%%%%%%%%%%%%%%%%%%%%%%%%%%%%%%%%%%%%%%%%%%%%%%%
%%%%%%%%%%%%%%%%%%%%%%%%%%%%%%%%%%%%%%%%%%%%%%%%%%%%%%%%%%%%%%%%%%%%%%%%%%%%%%%%%%%%%%%%%%%%%%%%%%%%%%%%
%%%%%%%%%%%%%%%%%%%%%%%%%%%%%%%%%%%%%%%%%%%%%%%%%%%%%%%%%%%%%%%%%%%%%%%%%%%%%%%%%%%%%%%%%%%%%%%%%%%%%%%%
%%%%%%%%%%%%%%%%%%%%%%%%%%%%%%%%%%%%%%%%%%%%%%%%%%%%%%%%%%%%%%%%%%%%%%%%%%%%%%%%%%%%%%%%%%%%%%%%%%%%%%%%
%%%%%%%%%%%%%%%%%%%%%%%%%%%%%%%%%%%%%%%%%%%%%%%%%%%%%%%%%%%%%%%%%%%%%%%%%%%%%%%%%%%%%%%%%%%%%%%%%%%%%%%%

\section{The automorphism group of the measure algebra}\label{sec:meas}

Let $(X, \Sigma ,\mu)$ be a standard probability space, e.g., $X=[0,1]$ with its Borel $\sigma$-algebra $\Sigma$ and Lebesgue measure $\mu$.  Let also ${\sf{Meas}}_\mu$ be the quotient of the $\sigma$-algebra $\Sigma$ by the ideal of null sets. We observe that ${\sf{Meas}}_\mu$ becomes a complete metric space under the metric 
$$
d\big([A],[B]\big)=\mu(A\triangle B),
$$
where $[A]$ denotes the equivalence class of an element $A\in \Sigma$. 
Observe also that the Boolean operations $\vee$, $\wedge$ and $\neg$ are continuous functions on the metric space ${\sf{Meas}}_\mu$ and hence the latter can be seen as a separable complete metric Boolean algebra structure. In particular, the automorphism group $\sf{Aut}(\sf{Meas}_\mu)$ is Polish in the topology of pointwise convergence with respect to the metric $d$.

\begin{lemme}\cite{sabok}
The measure algebra ${\sf Meas}_\mu$ has the Hrushovski property. 
\end{lemme}

\begin{proof}
Suppose $\ku A$ is a finite subalgebra of ${\sf Meas}_\mu$ and $a,b\in \ku A$. We say that $a$ and $b$ are {\em similarly partitioned by $\ku A$} provided that we can write their partitions into atoms of $\ku A$ as
$$
a=a_1\vee \cdots \vee a_n \qquad\text{and}\qquad b=b_1\vee\cdots\vee b_n
$$
so that $\mu(a_i)=\mu(b_i)$ for all $i$. In particular, $\mu(a)=\sum_i\mu(a_i)=\sum_i\mu(b_i)=\mu(b)$. 
Also, a pair  $\ku A\subseteq \ku B$ of finite subalgebras is {\em good} if any two atoms $a,b$ of $\ku A$ with $\mu(a)=\mu(b)$ are similarly partitioned by $\ku B$. Observe that, if $a,b\in \ku A$ are similarly partitioned by $\ku A$ and $\ku A\subseteq \ku B$ is good, then $a,b$ are also similarly partitioned by $\ku B$.

\begin{claim}
Suppose $\ku A_0\subseteq {\sf Meas}_\mu$ is a finite subalgebra. Then there is a finite algebra $\ku A_0\subseteq \ku B\subseteq {\sf Meas}_\mu$ so that any two disjoint elements $a,b\in \ku A_0$ with $\mu(a)=\mu(b)$ are similarly partitioned by $\ku B$.
\end{claim}
To see this, list all pairs of disjoint elements of $\ku A_0$ with equal measure as $(a_1,b^1),\ldots,(a_p,b^p)$. By induction on $r\leqslant p$, we build finite algebras $\ku A_r\subseteq {\sf Meas}_\mu$ so that
\begin{enumerate}
\item[(i)] $\ku A_0\subseteq \ldots\subseteq \ku A_p$,
\item[(ii)] each pair $\ku A_{r-1}\subseteq \ku A_{r}$ is good,
\item[(iii)] $a_r$ and $b^r$ are similarly partitioned by $\ku A_r$.
\end{enumerate}
Assuming this can be done, we may then simply let $\ku B=\ku A_p$. 

For the construction, assume that $\ku A_{r-1}$ is given and list without repetition the atoms of $\ku A_{r-1}$  as
$$
x_1,\ldots,  x_n, y^1,\ldots,y^m, z_1,\ldots, z_l
$$
in such a way that, for some $0\leqslant k\leqslant \tfrac l2$, we have
$$
a_r=x_1\vee \cdots \vee x_n \vee z_1\vee \cdots\vee z_k\quad\text{and}\quad b^r=y_1\vee \cdots \vee y_m \vee z_{k+1}\vee \cdots\vee z_{2k},
$$
$\mu(z_i)=\mu(z_{k+i})$ for all $i\leqslant k$ and $\mu(x_i)\neq \mu(y^j)$ for all $i,j$.
Set
$$
x=x_{1}\vee \cdots \vee x_n \qquad\text{and}\qquad y=y^{1}\vee\cdots\vee y^m
$$
and let $R=\mu(x)=\mu(y)$.  For all $i,j$, we choose partitions
\begin{equation}\label{xy partition}
x_i=x_i^{1}\vee \cdots \vee x_i^m
\qquad\text{and}\qquad 
y^j=y^j_{1}\vee \cdots\vee y^j_n
\end{equation}
so that
$$
\mu(x_i^j)=\tfrac{\mu(y^j)}{\mu(y)}\mu(x_i)=\frac{\mu(x_i)\mu(y^j)}{R}=\tfrac{\mu(x_i)}{\mu(x)}\mu(y^j)=\mu(y_i^j).
$$
Also, if $e\in \{z_1,\ldots, z_l\}$ satisfies $\mu(e)=\mu(x_i)$ for some $i$, we choose a partition
\begin{equation}\label{x partition}
e=e^{1}\vee \cdots\vee e^m
\end{equation}
with $\mu(e^j)=\frac{\mu(x_i)\mu(y^j)}{R}$. Similarly, if $e\in \{z_1,\ldots, z_l\}$ satisfies $\mu(e)=\mu(y^j)$ for some $j$, we choose a partition
\begin{equation}\label{y partition}
e=e_{1}\vee \cdots\vee e_n
\end{equation}
with $\mu(e_i)=\frac{\mu(x_i)\mu(y^j)}{R}$. Note that, because $\mu(x_i)\neq \mu(y^j)$, only one of these can happen.
Finally, we may now let $\ku A_r$ be the algebra generated by $\ku A_{r-1}$ along with the new partitioning elements from \eqref{xy partition}, \eqref{x partition} and \eqref{y partition}. Then $\ku A_{r-1}\subseteq \ku A_{r}$ is good and $a_r, b^r$ are similarly partitioned by $\ku A_r$. 

Using the claim, we can now prove the Hrushovski property. Indeed, assume that we are given a finite collection $\{\phi_i\}_{i\in I}$ of isomorphisms 
$$
\ku C_i\maps{\phi_i}\ku D_i
$$
between finitely generated (and thus finite) subalgebras $\ku C_i,\ku D_i\subseteq {\sf Meas}_\mu$. 
Let $\ku A_0$ be the finite algebra generated by all of the $\ku C_i$ and $\ku D_i$ and choose $\ku B$ as in the claim. Then we may extend each $\phi_i$ to an automorphism $\hat\phi_i$ of $\ku B$. Indeed, if $c$ is an atom of $\ku C_i$, then $d=\phi_i(c)$ is an atom of $\ku D_i$. Then $c\wedge \neg d$ and $d\wedge \neg c$ are two disjoint elements of $\ku A_0$ of the same measure and hence are similarly partitioned
$$
c\wedge \neg d=x_1\vee \cdots \vee x_n \qquad\text{and}\qquad d\wedge \neg c=y_1\vee\cdots\vee y_n
$$
by $\ku B$. We may thus define $\hat\phi_i(x_j)=y_j$ for all $j$ and let $\hat\phi_i(x)=x$ for any atom $x$ of $\ku B$ with $x\leqslant c\wedge d$. Doing this for all atoms $c$ of $\ku C_i$ gives us the promised extension $\hat\phi_i$. To further extend $\hat\phi_i$ to an automorphism of ${\sf Meas}_\mu$, for any atom $b$ of $\ku B$, we extend $\hat\phi_i$ arbitrarily to an isomorphism between the measured algebras of ${\sf Meas}_\mu$ beneath the elements $b$ and $\hat\phi_i(b)$.
\end{proof}

\begin{lemme}\label{lem:measurealgebra}
Suppose $a_1\vee \cdots\vee a_n$, $b_1\vee \cdots\vee b_n$ and $c_1\vee \cdots\vee c_m$ are three partitions of the unit in ${\sf Meas}_\mu$ so that
$$
\mu(a_i)=\mu(b_i) \qquad\text{and}\qquad \mu(a_i\wedge\neg b_i)\leqslant \mu(a_i\wedge c_1)
$$
for all $i$.
Then there is a partition of the unit $e_1\vee \cdots\vee e_n$ so that, for all $i, j$, 
\begin{enumerate}
\item $\mu(e_i)=\mu(a_i)=\mu(b_i)$,
\item $\mu(e_i\wedge a_j)=\mu(b_i\wedge a_j)$,
\item $\mu(e_i\wedge c_j)= \mu(a_i\wedge c_j)$.
\end{enumerate}
\end{lemme}

\begin{proof}
For a given $i$, note that 
$$
\sum_{j\neq i}\mu(a_i\wedge b_j)= \mu(a_i\wedge\neg b_i)\leqslant \mu(a_i\wedge c_1),
$$
whereby we can find a partition 
$$
a_i\wedge c_1 = d_{i1}\vee \cdots\vee d_{in}
$$
satisfying
$$
\mu(d_{ij})=\mu(a_i\wedge b_j)
$$
for all $j\neq i$. We then set
$$
e_i= (a_i\wedge \neg c_1)\vee \bigvee_{ j}d_{ji}.
$$

Observe first that, as $d_{ji}\leqslant a_j\wedge c_1$, we have
$$
e_i\wedge a_i=(a_i\wedge \neg c_1)\vee d_{ii},
$$
whereas, for $j\neq i$,
$$
e_i\wedge a_j=d_{ji}.
$$
So 
\maths{
\mu(e_i\wedge a_i)
&=\mu (a_i\wedge \neg c_1)+\mu(d_{ii})\\
&=\mu (a_i\wedge \neg c_1)+\mu (a_i\wedge c_1)-\sum_{j\neq i}\mu (d_{ij})\\
&=\mu (a_i)-\sum_{j\neq i}\mu (a_i\wedge b_j)\\
&=\mu (a_i)-\mu (a_i\wedge \neg b_i)\\
&=\mu (b_i\wedge a_i)
}
and, for $j\neq i$,
\maths{
\mu(e_i\wedge a_j)
&=\mu (d_{ji})=\mu (b_i\wedge a_j).
}
Thus,
$$
\mu(e_i)=\sum_j\mu(e_i\wedge a_j)=\sum_j\mu(b_i\wedge a_j)=\mu(b_i)=\mu(a_i).
$$
Also, for $j\neq 1$, $e_i\wedge c_j=a_i\wedge c_j$ and so
\maths{
\mu(e_i\wedge c_j)
&=\mu(a_i\wedge c_j),
}
whereas
\maths{
\mu(e_i\wedge c_1)
&=\mu(e_i)-\sum_{j\neq 1}\mu(e_i\wedge c_j)\\
&=\mu(a_i)-\sum_{j\neq 1}\mu(a_i\wedge c_j)\\
&=\mu(a_i\wedge c_1).
}
\end{proof}

\begin{thm}\cite{BBM} The group ${\sf Aut}({\sf Meas}_\mu)$ has the automatic continuity property.    
\end{thm}
\begin{proof}
We show that the conditions of Corollary \ref{cor:autconti} are satisfied by the structure ${\sf Meas}_\mu$. First of all, observe that the collection $\ku{B}$ of all finite tuples $\ov{a} = (a_1, \dots, a_n)$ of finite non-zero partitions of the unit in ${\sf Meas}_\mu$ is a basis for ${\sf Aut}({\sf Meas}_\mu)$. Let $\ov{a} \in \ku{B}$ and $\epsilon >0$ be given. 
We claim that the open set
$$
U=\big\{    (g,f) \in N(\ov{a}, \epsilon) \times {\sf Aut}({\sf Meas}_\mu)   \del    \mu(a_i \wedge \neg fa_i) < \mu(a_i \wedge ga_1) \text{ for all }i\big\}
$$
is nonempty. To see this, pick some $g \in N(\ov{a}, \epsilon)$ such that $\mu(a_i \wedge ga_1) > 0$ for all $i$ and set 
$$
\delta =\min_i\mu(a_i \wedge ga_1).
$$
Then, if $f \in N(\ov{a}, \delta)$, we have, for all $i$,
$$
\mu(a_i \wedge \neg fa_i) \;\leqslant\; \mu(a_i \triangle fa_i)\; < \;\delta\; <\; \mu(a_i \wedge ga_1)
$$
and so $(g,f)\in U$. 

Furthermore, for any $(g,f)\in U$, write $\ov b=f\ov a$ and $\ov c=g\ov b$, whereby
$$
\mu(a_i)=\mu(b_i) \qquad\text{and}\qquad \mu(a_i \wedge \neg b_i) < \mu(a_i \wedge c_1)
$$
for all $i$.  By Lemma \ref{lem:measurealgebra} there is a partition of the unit $e_1 \vee \dots \vee e_n$ so that, for all $i,j$,
\begin{enumerate}
\item $\mu(e_i)=\mu(a_i)=\mu(b_i)$,
\item $\mu(e_i\wedge a_j)=\mu(b_i\wedge a_j)$,
\item $\mu(e_i\wedge c_j)= \mu(a_i\wedge c_j)$.
\end{enumerate}
Condition (2) entails $\ov{e} \in \ku{O}(\ov{b}/\ov{a})$, whereas condition (3) implies $\ov{e} \in \ku{O}(\ov{a}/\ov{c})$ and thus,
$$
\ov{e} \in \ku{O}(f\ov{a}/\ov{a}) \cap \ku{O}(\ov{a}/g\ov{a}).
$$
This shows that ${\sf Aut}({\sf Meas}_\mu)$ satisfies the criteria of Corollary \ref{cor:autconti} and therefore has the automatic continuity property.
\end{proof}

%%%%%%%%%%%%%%%%%%%%%%%%%%%%%%%%%%%%%%%%%%%%%%%%%%%%%%%%%%%%%%%%%%%%%%%%%%%%%%%%%%%%%%%%%%%%%%%%%%%%%%%%
%%%%%%%%%%%%%%%%%%%%%%%%%%%%%%%%%%%%%%%%%%%%%%%%%%%%%%%%%%%%%%%%%%%%%%%%%%%%%%%%%%%%%%%%%%%%%%%%%%%%%%%%
%%%%%%%%%%%%%%%%%%%%%%%%%%%%%%%%%%%%%%%%%%%%%%%%%%%%%%%%%%%%%%%%%%%%%%%%%%%%%%%%%%%%%%%%%%%%%%%%%%%%%%%%
%%%%%%%%%%%%%%%%%%%%%%%%%%%%%%%%%%%%%%%%%%%%%%%%%%%%%%%%%%%%%%%%%%%%%%%%%%%%%%%%%%%%%%%%%%%%%%%%%%%%%%%%
%%%%%%%%%%%%%%%%%%%%%%%%%%%%%%%%%%%%%%%%%%%%%%%%%%%%%%%%%%%%%%%%%%%%%%%%%%%%%%%%%%%%%%%%%%%%%%%%%%%%%%%%
%%%%%%%%%%%%%%%%%%%%%%%%%%%%%%%%%%%%%%%%%%%%%%%%%%%%%%%%%%%%%%%%%%%%%%%%%%%%%%%%%%%%%%%%%%%%%%%%%%%%%%%%
%%%%%%%%%%%%%%%%%%%%%%%%%%%%%%%%%%%%%%%%%%%%%%%%%%%%%%%%%%%%%%%%%%%%%%%%%%%%%%%%%%%%%%%%%%%%%%%%%%%%%%%%
%%%%%%%%%%%%%%%%%%%%%%%%%%%%%%%%%%%%%%%%%%%%%%%%%%%%%%%%%%%%%%%%%%%%%%%%%%%%%%%%%%%%%%%%%%%%%%%%%%%%%%%%
%%%%%%%%%%%%%%%%%%%%%%%%%%%%%%%%%%%%%%%%%%%%%%%%%%%%%%%%%%%%%%%%%%%%%%%%%%%%%%%%%%%%%%%%%%%%%%%%%%%%%%%%
%%%%%%%%%%%%%%%%%%%%%%%%%%%%%%%%%%%%%%%%%%%%%%%%%%%%%%%%%%%%%%%%%%%%%%%%%%%%%%%%%%%%%%%%%%%%%%%%%%%%%%%%

\section{The unitary group of the separable Hilbert space}\label{sec:hilbert}
Let $H$ be a separable infinite-dimensional complex Hilbert space and let $U(H)$ be its unitary group equipped with the strong operator topology, that is, the topology of pointwise norm-convergence on $H$.  Let $\go K\subseteq \C$ be a countable algebraically closed subfield and consider the following first-order language,
$$
L = L_{\sf dist} \cup \{0, +\}\cup \{m_q \del  q \in \go K\}\cup\{  R_r, I_r \del  r \in \Q\},
$$
where the $R_r$ and  $I_r$ are binary relation symbols, $m_r$ are unary function symbols, $+$ is a binary function symbol, and $0$ is a constant symbol. Consider now $H$ as a complete metric structure in the language $L$  with interpretations
\begin{align*}
 & H \models D_r(x,y) \quad\Leftrightarrow\quad \norm{x-y} \leqslant r, \\
    & H \models R_r(x,y) \quad\Leftrightarrow\quad {\mathfrak {Re}}\langle x, y \rangle \leqslant r, \\
    & H \models I_r(x,y) \quad\Leftrightarrow\quad  {\mathfrak {Im}}\langle x, y \rangle \leqslant r, \\
    & m^H_q(x) = qx, \\
    & +^{ H}(x,y) = x+y, \\
    & 0^{ H} = 0.
\end{align*}
Note that ${\sf Aut}(H)$ is isomorphic to $U(H)$, because $\go K$-linear mappings that preserve the inner product are just unitary operators. We remark that by simple arguments using orthogonality and the Gram--Schmidt orthonormalisation procedure, $H$ can be seen to have the extension and Hrushovski properties (see, e.g., \cite[Section 6]{ros-forum}).

\begin{lemme}
 Suppose $K$ is a closed infinite- and coinfinite-dimensional subspace of the Hilbert space $H$. Then there is a closed subgroup $G\leqslant U(H)$ so that $G\iso {\sf Aut}({\sf Meas}_\mu)$ and an isomorphic embedding $U(K)\maps \Delta G$ for which
$$
\Delta(u)|_K=u
$$ 
for all $u\in U(K)$.
\end{lemme}
 
\begin{proof}
Following the exposition of \cite[Appendix D]{kechris-global}, there is a standard probability space $(X,\mu)$, a closed infinite- and coinfinite-dimensional subspace $L\subseteq L^2(X,\mu)$ and a isomorphic embedding
$$
U(L)\maps\Theta{\sf Aut}(X,\mu)
$$
so that, if 
$$
{\sf Aut}(X,\mu) \maps \Phi U\big(L^2(X,\mu)\big)
$$
denotes the Koopman representation, i.e.,
$$
\Phi(f)(\phi)=\phi\circ f\inv,
$$
then $(\Phi\circ\Theta)u|_L=u$ for all $u\in U(L)$. Let now $H\maps TL^2(X,\mu)$ be a unitary isomorphism so that $T[K]=L$ and set 
$$
G=\big\{T\inv\circ\Phi(f)\circ T\del f\in {\sf Aut}(X,\mu)\big\}.
$$
Because the Koopman representation embeds ${\sf Aut}(X,\mu)$ into $U\big(L^2(X,\mu)\big)$, we have $G\iso {\sf Aut}({\sf Meas}_\mu)$. Finally, for $u\in U(K)$, let
$$
\Delta(u)=T\inv\circ     \Phi\big(\Theta(TuT\inv|_L)     \big)           \circ T,
$$
whereby $\Delta(u)|_K=u$.
\end{proof}

\begin{thm}\cite{tsankov}
The unitary group $U(H)$ of separable infinite-dimensional Hilbert space $H$ has the automatic continuity property.
\end{thm}

\begin{proof}
Fix a symmetric $\sigma$-syndetic subset $W\subseteq U(H)$.
By Lemma \ref{lem:sabok}, there is an orthonormal tuple $\ov a=(a_1,\ldots, a_n)$ in $H$ so that the pointwise stabiliser
$$
U(H,\ov a)=\{u\in U(H)\del u(a_i)=a_i \text{ for all }i\}
$$
is contained in $W^{10}$. Let also $K\subseteq H$ be a closed  infinite- and coinfinite-dimensional subspace of $H$ containing the vectors $a_1.\ldots, a_n$ and fix a closed subgroup $G\leqslant U(H)$ so that $G\iso {\sf Aut}({\sf Meas}_\mu)$ and an isomorphic embedding $U(K)\maps \Delta G$ for which
$$
\Delta(u)|_K=u
$$ 
for all $u\in U(K)$. Because $W$ is $\sigma$-syndetic in $U(H)$, $V=W^2\cap G$ is symmetric $\sigma$-syndetic in $G$ and hence, since ${\sf Aut}({\sf Meas}_\mu)$ is Steinhaus with exponent $64$, the sets 
$$
V^{64}\subseteq W^{128}\cap G
$$
are identity neighbourhoods in $G$. By continuity of $\Delta$, the inverse image $\Delta\inv(W^{128})$
is an identity neighbourhood in $U(K)$ and hence, for some orthonormal tuple $\ov b=(b_1,\ldots, b_m)$ extending $\ov a$ and some $\eps>0$, we find that
$$
M=\{u\in U(K)\del \norm{u(b_i)-b_i}<\eps \text{ for all }i\}\subseteq  \Delta\inv(W^{128}).
$$
We claim that
$$
N=\{u\in U(H)\del \norm{u(b_i)-b_i}<\eps \text{ for all }i\}\subseteq  W^{158}.
$$
To see this, fix some $u\in N$ and expand $\ov b$ to an orthonormal tuple
$$
b_1,\ldots, b_m, c_1,\ldots, c_k
$$
in $H$ so that $u(b_i)\in {\sf span}(b_1,\ldots, b_m, c_1,\ldots, c_k)$ for all $i$. Now find vectors $d_1,\ldots, d_k\in K$ so that also
$$
b_1,\ldots, b_m, d_1,\ldots, d_k
$$
is orthonormal and let $v\in U(H, \ov b)\leqslant U(H,\ov a)\subseteq W^{10}$ be chosen so that $v(c_i)=d_i$ for all $i$. Then, for all $i$,
$$
vu(b_i)\in {\sf span}(b_1,\ldots, b_m, d_1,\ldots, d_k)\subseteq K,
$$
i.e., $b_i\in K\cap (vu)\inv(K)$.
It follows that there is some $w\in U(H, \ov b)\subseteq W^{10}$ so that
$$
w [K]=(vu)\inv (K)
$$
and thus $vuw[K]=K$. Furthermore, for all $i$,
$$
\norm{vuw(b_i)-b_i}=\norm{uw(b_i)-v\inv(b_i)}=\norm{u(b_i)-b_i}<\eps,
$$ 
which shows that $vuw|_K\in M$ and $\Delta(vuw|_K)\in W^{128}$. As $\Delta\big(vuw|_K\big)|_K=vuw|_K$, we finally see that
$$
vuw\in \Delta\big(vuw|_K\big)\cdot U(H, \ov b) \subseteq W^{128}\cdot W^{10},
$$
whereby
$$
u\in v\inv \cdot W^{138}\cdot w\inv \subseteq W^{158}
$$
as claimed. Because $N$ is an identity neighbourhood in $U(H)$, the latter group is Steinhaus and thus satisfies the automatic continuity property.
\end{proof}

%%%%%%%%%%%%%%%%%%%%%%%%%%%%%%%%%%%%%%%%%%%%%%%%%%%%
%%%%%%%%%%%%%%%%%%%%%%%%%%%%%%%%%%%%%%%%%%%%%%%%%%%%
%%%%%%%%%%%%%%%%%%%%%%%%%%%%%%%%%%%%%%%%%%%%%%%%%%%%
%%%%%%%%%%%%%%%%%%%%%%%%%%%%%%%%%%%%%%%%%%%%%%%%%%%%
%%%%%%%%%%%%%%%%%%%%%%%%%%%%%%%%%%%%%%%%%%%%%%%%%%%%

%%%%%%%%%%%%%%%%%%%%%%%%%%%%%%%%%%%%%%%%%%%%%%%%%%%%%%%%%%%%%%%%%%%%%%%
%%%%%%%%%%%%%%%%%%%%%%%%%%%%%%%%%%%%%%%%%%%%%%%%%%%%%%%%%%%%%%%%%%%%%%%
%%%%%%%%%%%%%%%%%%%%%%%%%%%%%%%%%%%%%%%%%%%%%%%%%%%%%%%%%%%%%%%%%%%%%%%
%%%%%%%%%%%%%%%%%%%%%%%%%%%%%%%%%%%%%%%%%%%%%%%%%%%%%%%%%%%%%%%%%%%%%%%
%%%%%%%%%%%%%%%%%%%%%%%%%%%%%%%%%%%%%%%%%%%%%%%%%%%%%%%%%%%%%%%%%%%%%%%
%%%%%%%%%%%%%%%%%%%%%%%%%%%%%%%%%%%%%%%%%%%%%%%%%%%%%%%%%%%%%%%%%%%%%%%
\end{document}